\documentclass[10pt, a4paper]{article}
\usepackage[a4paper, hmargin={2.7cm,2.7cm},vmargin={3.3cm,3.3cm}]{geometry}
\usepackage[bf, compact, pagestyles, small]{titlesec}
\titlespacing*{\section}{0pt}{14pt}{4pt}
\titlespacing*{\subsection}{0pt}{8pt}{3pt}
\usepackage[english]{babel}
\usepackage{color}
\usepackage{cite}
\usepackage{graphicx}
\usepackage{caption}
\usepackage{subfig}
\usepackage{amsmath}
\usepackage{amssymb}
\usepackage{amsfonts}
\usepackage{amsthm}
\usepackage{algorithm}
\usepackage{enumerate}
\usepackage{algpseudocode}
\usepackage[nottoc]{tocbibind}
\usepackage{amsthm}
\usepackage[T1]{fontenc}
\usepackage{dsfont}
\usepackage{fancyhdr}
\usepackage{makeidx}

\usepackage{etoolbox}

\makeatletter
\patchcmd{\ttlh@hang}{\parindent\z@}{\parindent\z@\leavevmode}{}{}
\patchcmd{\ttlh@hang}{\noindent}{}{}{}
\makeatother

\AtEndDocument{\bigskip{%
  \textsc{University of Vienna, Oskar-Morgenstern-Platz 1, 1090 Vienna, Austria.} \par  
  \textit{E-mail address}: \texttt{jordy-timo.van-velthoven@univie.ac.at} \par
}}

\newcommand\numberthis{\addtocounter{equation}{1}\tag{\theequation}}

\newtheorem{theorem}{Theorem}[section]
\newtheorem{lemma}[theorem]{Lemma}
\newtheorem{proposition}[theorem]{Proposition}
\newtheorem{corollary}[theorem]{Corollary}

\theoremstyle{definition}
\newtheorem{definition}[theorem]{Definition}
\newtheorem{examplex}[theorem]{Example}
\newenvironment{example}
  {\pushQED{\qed}\examplex}
  {\popQED\endexamplex}
 
\theoremstyle{remark}
\newtheorem{remark}[theorem]{Remark}

\numberwithin{equation}{section}

\newcommand\blfootnote[1]{%
  \begingroup
  \renewcommand\thefootnote{}\footnote{#1}%
  \addtocounter{footnote}{-1}%
  \endgroup
}

\DeclareMathOperator*{\covolh}{d_{H}}
\DeclareMathOperator*{\vol}{vol}

\DeclareMathOperator*{\loc}{loc}
\DeclareMathOperator*{\covol}{d}
\DeclareMathOperator*{\rel}{Rel}
\DeclareMathOperator*{\supp}{supp}
\DeclareMathOperator*{\diag}{diag}

\DeclareMathOperator*{\essinf}{ess\,inf}

\begin{document}

\author{Jordy Timo van Velthoven}
\title{On the local integrability condition for generalised translation-invariant systems}
\date{}

\maketitle


\begin{abstract}
This paper considers the local integrability condition for generalised translation-invariant systems and its relation to the Calder\'on integrability condition, the temperateness condition and the uniform counting estimate. It is shown that sufficient and necessary conditions for satisfying the local integrability condition are closely related to lower and upper bounds on the number of lattice points that intersect with the translates of a compact set. The results are complemented by examples that illustrate the crucial interplay between the translation subgroups and the generating functions of the system.
\end{abstract}

\section{Introduction} \label{sec: intro}
Generalised shift-invariant systems form a large class of structured function systems. In the setting of a locally compact Abelian group $G$, given a countable collection $ \{g_{j} \}_{j \in J}$ of functions $ g_{j}  \in L^2 (G)$ and a collection $(\Gamma_j)_{j \in J}$ of discrete, co-compact subgroups $\Gamma_j \subseteq G$, a \emph{generalised shift-invariant system} in $L^2 (G)$ is a system of the form
\[ \bigcup_{j \in J} \bigg\{g_{j} (\cdot - \gamma) \; : \; \gamma \in  \Gamma_j  \bigg\}. \]
A generalised shift-invariant system $\cup_{j \in J} \{g_{j}(\cdot - \gamma) \}_{\gamma \in \Gamma_j}$ in $L^2 (G)$ is called a \emph{generalised shift-invariant frame}, or simply a \emph{frame}, for $L^2 (G)$ if there exist constants $A, B > 0$ such that 
\[
A \| f \|_2^2 \leq \sum_{j \in J} \sum_{\gamma \in \Gamma_j} |\langle f, g_{j} (\cdot - \gamma) \rangle |^2  \;  \leq B \| f \|^2_2
\]
for all $f \in L^2 (G)$. Given a frame $\cup_{j \in J} \{g_j (\cdot - \gamma) \}_{\gamma \in \Gamma_j}$ for $L^2 (G)$, any function $f \in L^2 (G)$ can be represented by an unconditionally, norm convergent series of the form $f = \sum_{j \in J} \sum_{\gamma \in \Gamma_j} c_{j, \gamma}  g_j(\cdot - \gamma)$ for some coefficients $(c_{j, \gamma})_{j \in J, \gamma \in \Gamma_j}$. 

The frame properties of arbitrary generalised shift-invariant systems in $L^2 (\mathbb{R}^d)$ have first been investigated by Hern\'andez, Labate \& Weiss \cite{hernandez2002unified} and Ron \& Shen \cite{ron2005generalized}. Afterwards, generalised shift-invariant frames have been the topic  of numerous papers including \cite{jakobsen2016reproducing,christensen2005generalized,christensen2017explicit,kutyniok2006theory,fuhr2017system,lemvig2017sufficient}. In studying the frame properties of a generalised shift-invariant system, it is usually assumed that the system $\cup_{j \in J} \{g_{j}(\cdot - \gamma) \}_{\gamma \in \Gamma_j}$ satisfies, for any compact set $K$ in the Fourier domain $\widehat{G}$, 
\begin{align} \label{eq:LIC_intro}
\sum_{j \in J} \frac{1}{\vol (G / \Gamma_j)} \sum_{\alpha \in \Gamma_j^{\perp}} \int_{K \cap  \alpha^{-1} K} | \hat{g}_{j} (\omega) |^2 \; d\mu_{\widehat{G}} (\omega) < \infty.
\end{align}
The condition \eqref{eq:LIC_intro} is known as the \emph{local integrability condition} (LIC). In essence, the local integrability condition constraints the interaction between the translation subgroups and the generating functions of the system. The reason for imposing the local integrability condition is twofold. On the one hand, the local integrability condition guarantees the almost periodicity of an auxiliary function that is useful in studying frame properties of a generalised shift-invariant system, see \cite{laugesen2001completeness,laugesen2002translational,balan2014multi} for approaches based on similar techniques. Here, it should be noted that there are weaker conditions than the LIC that still guarantee this almost periodicity, namely the so-called \emph{$\infty$-unconditional convergence property} \cite{fuhr2017system} and the \emph{$\alpha$-local integrability condition} \cite{jakobsen2016reproducing}. Another reason for imposing the local integrability condition, or a similar condition, is that without such a condition structural results that substantiate intuition might fail. For example, a characterisation of generalised shift-invariant Parseval frames \cite{bownik2004spectral} and necessary conditions involving the Calder\'on sum or the system bandwidth \cite{fuhr2017system} fail without assuming the local integrability condition or a similar condition.

Verifying the local integrability condition is usually a non-trivial task and might even be an obstacle. However, in special cases simple sufficient conditions for the local integrability condition are known. For example, a simple characterisation of the local integrability condition for arbitrary generalised shift-invariant systems on the real line and on locally compact Abelian groups possessing a compact connected component were given in \cite{christensen2017explicit} and \cite{kutyniok2006theory}, respectively. For special cases of generalised translation-invariant systems such as wavelet systems and wave packet systems simple sufficient conditions or characterisations are contained in \cite{bownik2011affine,kutyniok2006local} and \cite{labate2004approach}, respectively.

In this paper the local integrability condition is considered for arbitrary generalised shift-invariant systems on locally compact Abelian groups. The considered systems will be more general than the ones mentioned above. Following \cite{iverson2015subspaces,jakobsen2016reproducing}, the index sets for the generating functions are namely allowed to be uncountable. Moreover, the translation subgroups of the system are allowed to be non-discrete as in \cite{bownik2015structure,jakobsen2016reproducing, barbieri2015zak}. To stress that the involved subgroups need not be discrete the term \emph{translation-invariant}, rather than shift-invariant, is adopted. The results focus on the subtle interplay between the local integrability condition and three related conditions involving the generating functions and the translation subgroups, namely the Calder\'on integrability condition, the temperateness condition and the uniform counting estimate; see Section 2.2 below for precise definitions. These conditions have appeared implicitly or explicitly in the literature before, but their interrelation and their relation to the local integrability condition has not been systematically addressed. The results in this paper include simple sufficient and necessary conditions for the local integrability condition that involve the Calder\'on integrability condition, the temperateness condition and the uniform counting estimate. Indeed, as special cases, the aforementioned characterisations of the local integrability condition obtained in the papers \cite{christensen2017explicit,kutyniok2006theory} are recovered. To be more precise, the characterisation of the local integrability condition on locally compact Abelian groups with a compact connected component given in \cite{kutyniok2006theory} is extended to generalised translation-invariant systems for which the translation subgroups are not necessarily discrete. A higher-dimensional analogue of the characterisation of the local integrability condition on the real line \cite{christensen2017explicit} is also obtained. Aside these results, the paper includes several examples illustrating the subtle interplay between the translation subgroups and the generating functions. Noteworthy is here especially Example \ref{ex:mainexample}, which demonstrates that the temperateness condition is in general not necessary for the local integrability condition to be satisfied. This contrasts sharply with the situation on some specific groups or for large classes of generalised translation-invariant systems. 

The paper is organised as follows. Generalised translation-invariant systems and associated technical conditions are introduced in Section 2. The subtle interplay between these conditions forms the topic of Section 3. In this section sufficient and necessary conditions for the local integrability condition are given. Section 4 considers the local integrability condition for the cases in which the underlying group possesses a compact open subgroup or is a Euclidean space. 

\section{Generalised translation-invariant systems and associated conditions}
In this section the key notions of the paper are introduced. The first subsection sets up the general notation. 
Generalised translation-invariant systems and frames are introduced in Section \ref{sec:GTI_LCA}. Several technical conditions associated with generalised translation-invariant systems are considered in Section \ref{sec:LIC_counting}. 

\subsection{Notation and normalisations}
Let $G$ denote a locally compact Abelian group. It will be assumed that $G$ is Hausdorff and satisfies the second axiom of countability, which is equivalent to $G$ being metrisable and $\sigma$-compact. The Pontryagin dual  of $G$ will be denoted by $\widehat{G}$ and forms a second countable locally compact Abelian group itself. The group operations in $G$ and $\widehat{G}$ are written additively $+$ and multiplicatively $\cdot$, respectively. Correspondingly, the elements $0$ and $1$ denote the identity elements in $G$ respectively $\widehat{G}$. For two locally compact Abelian groups $G_1$ and $G_2$ that are isomorphic as topological groups, the notation $G_1 \cong G_2$ is used. 

The Haar measure on $G$ is always assumed to be given and is denoted by $\mu_G$. For a closed subgroup $H \subseteq G$, the corresponding quotient $G / H$ is a locally compact Abelian group as well. The groups $G$, $H$ and $G/H$ carry Haar measures $\mu_G$, $\mu_H$ and $\mu_{G/H}$, respectively. If two out of the three Haar measures $\mu_G$, $\mu_H$ and $\mu_{G/H}$ are given, then the third one can be normalised such that, for all $f \in L^1 (G)$,
\begin{align} \label{eq:weil}
\int_G f(x) \; d\mu_G (x) = \int_{G/H} \int_H f(x+h) \; d\mu_{H} (h) d\mu_{G/H} (\dot{x}),
\end{align}
where $\dot{x} = x + H$ with $x \in G$. The formula \eqref{eq:weil} is known as \emph{Weil's integral formula} and it is always assumed that the measures $\mu_G$, $\mu_H$ and $\mu_{G/H}$ are normalised such that \eqref{eq:weil} holds. The measures are then said to be \emph{canonically related}.

The Haar measure on the dual group $\widehat{G}$ is taken to be the Plancherel measure relative to the given measure $\mu_G$ on $G$. The annihilator $H^{\perp}$ of a closed subgroup $H \subseteq G$ is the closed subgroup $H^{\perp} := \{\omega \in \widehat{G} \; | \; \omega (x) = 1, \; \forall x \in H\}$ of $\widehat{G}$. The Haar measures on $\widehat{H} \cong \widehat{G} / H^{\perp}$ and $\widehat{G / H} \cong H^{\perp}$ are assumed to be the Plancherel measures relative to the measures $\mu_H$ and $\mu_{G/H}$, respectively. The Haar measures $\mu_{\widehat{G}}$, $\mu_{H^{\perp}}$ and $\mu_{\widehat{G} / H^{\perp}}$ chosen in this manner are canonically related. For other standard facts on locally compact Abelian groups, the reader is referred to \cite{hewitt1963abstract, rudin1962fourier}. 

\subsection{Generalised translation-invariant systems} \label{sec:GTI_LCA}

\begin{definition}
Let $J$ be a countable index set. For each $j \in J$, let $\Gamma_j \subseteq G$ be a closed, co-compact subgroup and let $P_j$ be an arbitrary index set. Let $\cup_{j \in J} \{g_{j,p} \}_{p \in P_j}$ be a family in $L^2 (G)$. The system of translates 
\[ \bigcup_{j \in J} \bigg\{ T_{\gamma} g_{j,p} \; : \; \gamma \in \Gamma_j, \; p \in P_j \bigg\} 
= \bigcup_{j \in J} \bigg\{  g_{j,p} (\cdot - \gamma) \; : \; \gamma \in \Gamma_j, \; p \in P_j \bigg\} \]
is called a \emph{generalised translation-invariant (GTI) system} in $L^2 (G)$. 
\end{definition}

For each translation subgroup $\Gamma_j$, it is assumed that the Haar measure $\mu_{\Gamma_j}$  is given. In this case there exists a unique quotient measure $\mu_{G / \Gamma_j}$ provided by Weil's integral formula. Using this quotient measure, the \emph{covolume} of $\Gamma_j \subseteq G$ is defined as
\[ \covol(\Gamma_j) = \mu_{G / \Gamma_j} (G / \Gamma_j). \]
The covolume $\covol (\Gamma_j)$ is finite precisely when $\Gamma_j \subseteq G$ is co-compact. In the special case of a uniform lattice $\Gamma_j \subseteq G$, e.g., a discrete, co-compact subgroup, the covolume coincides with the measure of a Borel transversal $X_j \subseteq G$ of $\Gamma_j$ in $G$ provided that $\Gamma_j$ is equipped with the counting measure. Unless otherwise specified, throughout this paper, any uniform lattice is assumed to be equipped with the counting measure. 

The generating functions of a generalised translation-invariant system are assumed to satisfy the following mild conditions, which will be referred to as the \emph{standing hypotheses} in accordance with \cite{jakobsen2016reproducing}. Here, the symbol $\mathcal{B}_X$ denotes the Borel $\sigma$-algebra on a topological space $X$.

\paragraph{\textbf{Standing hypotheses.}} For each $j \in J$,
\begin{enumerate}[(I)]
\item The triple $(P_j, \Sigma_{P_j}, \mu_{P_j})$ forms a $\sigma$-finite measure space;
\item The mapping $(P_j, \Sigma_{P_j}) \to (L^2 (G), \mathcal{B}_{L^2
    (G)}), \; p \mapsto g_{j,p}$ is $\Sigma_{P_j}$-measurable;
\item The mapping $(P_j \times G, \Sigma_{P_j} \otimes \mathcal{B}_G)
  \to (\mathbb{C}, \mathcal{B}_{\mathbb{C}}), \; (p, x) \mapsto
  g_{j,p} (x)$ is $(\Sigma_{P_j} \otimes \mathcal{B}_G)$-measurable.
\end{enumerate}

\begin{remark}
The standing hypotheses on the generating functions are automatically satisfied whenever the index sets $P_j$, $j \in J$, are locally compact Hausdorff spaces equipped with a Borel $\sigma$-algebra and a Radon measure and each map $p \mapsto g_{j,p}$ is continuous. Hence, in particular, the standing hypotheses are satisfied whenever each $P_j$, $j \in J$, is countable and equipped with the discrete $\sigma$-algebra and a (weighted) counting measure. 
\end{remark}

By the standing hypotheses, the integrals in the next definition are well-defined.

\begin{definition}
A generalised translation-invariant system $\cup_{j \in J} \{T_{\gamma} g_{j,p} \}_{\gamma \in \Gamma_j, p \in P_j}$ is called a \emph{generalised translation-invariant frame} for $L^2 (G)$ if there exist constants $A, B > 0$, called \emph{frame bounds}, such that 
\begin{align*}
A \| f \|_2^2 \leq \sum_{j \in J} \int_{P_j} \int_{\Gamma_j} | \langle f, T_{\gamma} g_{j,p} \rangle |^2 \; d\mu_{\Gamma_j} (\gamma) d\mu_{P_j} (p) \leq B \|f\|_2^2
\end{align*}
for all $f \in L^2 (G)$. A system $\cup_{j \in J} \{T_{\gamma} g_{j,p} \}_{\gamma \in \Gamma_j, p \in P_j}$ satisfying the upper bound is called a \emph{Bessel family}. 
\end{definition}

For a generalised translation-invariant system  $\cup_{j \in J} \{T_{\gamma} g_{j,p} \}_{\gamma \in \Gamma_j, p \in P_j}$ forming a frame, there always exists a family $\cup_{j \in J} \{\tilde{g}_{j,p,\gamma} \}_{\gamma \in \Gamma_j, p \in P_j}$ in  $L^2 (G)$ giving rise to reproducing formulae, given weakly by
\[
f = \sum_{j \in J} \int_{P_j} \int_{\Gamma_j}  \langle f, \tilde{g}_{j,p,\gamma} \rangle T_{\gamma} g_{j,p} \; d\mu_{\Gamma_j} (\gamma) d\mu_{P_j} (p) = \sum_{j \in J} \int_{P_j} \int_{\Gamma_j}  \langle f, T_{\gamma} g_{j,p} \rangle  \tilde{g}_{j,p,\gamma}\; d\mu_{\Gamma_j} (\gamma) d\mu_{P_j} (p)
\]
for all $f \in L^2 (G)$. In particular, if the frame bounds of a frame can be chosen to be one, then the family  $\cup_{j \in J} \{\tilde{g}_{j,p,\gamma} \}_{\gamma \in \Gamma_j, p \in P_j}$ can be chosen to be $\cup_{j \in J} \{ T_{\gamma} g_{j,p} \}_{\gamma \in \Gamma_j, p \in P_j}$.

For more on frame theory, the reader is referred to Christensen's book \cite{christensen2016introduction}. 

\subsection{Integrability conditions and the uniform counting estimate} \label{sec:LIC_counting}

Let $\mathcal{E}$ denote the collection of all closed Borel sets $E \subseteq \widehat{G}$ with Haar measure zero. For an $E \in \mathcal{E}$, define the subspace $\mathcal{D}_E (G)$ of $L^2 (G)$ as
\[  \mathcal{D}_E (G) = \bigg\{f \in L^2 (G) \; \bigg| \; \hat{f} \in L^{\infty}_c (\widehat{G}), \; \supp \hat{f} \subseteq \widehat{G} \setminus E \bigg\}, \]
where $\hat{f}$ denotes the Fourier-Plancherel transform of $f \in L^2 (G)$. 
The space $\mathcal{D}_E (G)$ is translation-invariant and norm dense in $L^2 (G)$.

Phrased in terms of $\mathcal{D}_E (G)$, the local integrability condition is defined as follows.

\begin{definition}
A generalised translation-invariant system $\cup_{j \in J} \{T_{\gamma} g_{j,p} \}_{\gamma \in \Gamma_j, p \in P_j}$ is said to satisfy the \emph{local integrability condition} (LIC), with respect to $E \in \mathcal{E}$, if
\begin{align} \label{eq:LIC}
 \sum_{j \in J} \frac{1}{\covol (\Gamma_j)} \int_{P_j} \sum_{\alpha
   \in \Gamma_j^{\perp}} \int_{\supp \hat{f}} | 
 \hat{f}(\omega  \alpha)  \hat{g}_{j,p} (\omega)|^2 \;  d\mu_{\widehat{G}} (\omega) d\mu_{P_j} (p) <
 \infty 
\end{align}
for all $f \in \mathcal{D}_E (G)$. 
\end{definition}

Equivalently, a system $\cup_{j \in J} \{T_{\gamma} g_{j,p} \}_{\gamma \in \Gamma_j, p \in P_j}$ satisfies the local integrability condition with respect to $E \in \mathcal{E}$ if
\[
 \sum_{j \in J} \frac{1}{\covol (\Gamma_j)} \int_{P_j} \sum_{\alpha
   \in \Gamma_j^{\perp}} \int_{K \cap \alpha^{-1} K} | 
  \hat{g}_{j,p} (\omega)|^2 \;  d\mu_{\widehat{G}} (\omega) d\mu_{P_j} (p) <
 \infty 
\]
for all compact sets $K \subseteq \widehat{G} \setminus E$.

The local integrability condition depends on the set $E \in \mathcal{E}$, which is sometimes called the \emph{blind spot} of the system \cite{fuhr2015coorbit}. Intuitively, the blind spot $E$ can be chosen to consist of points in the Fourier domain which the associated system fails to resolve. In Euclidian space $G = \mathbb{R}^d$, the blind spot is often chosen to be $E = \emptyset$ or $E = \{0\}$. Clearly, if the local integrability condition \eqref{eq:LIC} is satisfied with respect to $E = \emptyset$, then it is satisfied with respect to any $E \in \mathcal{E}$. 

\begin{definition}
A generalised translation-invariant system $\cup_{j \in J} \{T_{\gamma} g_{j,p} \}_{\gamma \in \Gamma_j, p \in P_j}$ is said to satisfy the \emph{Calder\'on integrability condition}, with respect to $E \in \mathcal{E}$, if
\begin{align} \label{eq:CalderonLI}
\sum_{j \in J}  \frac{1}{\covol(\Gamma_j)} \int_{P_j}  |\hat{g}_{j,p} (\cdot )|^2 \; d\mu_{P_j} (p) \in L^1_{\loc} (\widehat{G} \setminus E).
\end{align}
\end{definition}

It is customary to call the term $\sum_{j \in J}  \covol(\Gamma_j)^{-1} \int_{P_j}  |\hat{g}_{j,p} (\cdot )|^2 \; d\mu_{P_j} (p)$ the \emph{Calder\'on sum} or the \emph{Calder\'on integral} in accordance with wavelet theory \cite{chui2002characterization, fuhr2010generalized}.  For example, see \eqref{eq:Calderon_wavelets} below for the commonly known Calder\'on sum of a discrete wavelet system. 

In contrast to the local integrability condition \eqref{eq:LIC} and the Calder\'on integrability condition \eqref{eq:CalderonLI}, the next condition involves solely the generating functions of the system. 

\begin{definition}
Let $\cup_{j \in J} \{g_{j,p} \}_{p \in P_j}$  be a family of functions $g_{j,p} \in L^2 (G)$ satisfying the standing hypotheses. The family $\cup_{j \in J} \{g_{j,p} \}_{p \in P_j}$ is said to be \emph{strictly temperate}, 
with respect to $E \in \mathcal{E}$, if 
\begin{align} \label{eq:norm-bounded}
 \sum_{j \in J} \int_{P_j} |\hat{g}_{j,p} (\cdot )|^2 \; d\mu_{P_j} (p) \in L^1_{\loc} (\widehat{G} \setminus E).
\end{align}
\end{definition}

By Lemma \ref{lem:temperate_wavelet} below, the strictly temperateness condition is automatically satisfied for regular wavelet systems and for this reason it occurs only implicitly in wavelet theory.  The condition is, however, non-trivial for arbitrary generalised translation-invariant systems and occurs in \cite{christensen2017explicit, ron2005generalized}. More precisely, the condition \eqref{eq:norm-bounded} is a special case of the \emph{tempered} condition \cite[Definition 4.1]{ron2005generalized} whose terminology is adopted here. 

The last condition that will be introduced is the uniform counting estimate. As a motivation, it is instructive to relate it to the notion of relatively separateness. Recall that a discrete set $\Lambda \subseteq \widehat{G}$ is called \emph{relatively separated} if, for all relatively compact neighbourhoods of the identity $V$ with non-empty interior, 
\begin{align} \label{eq:relatively_separated}
\rel(\Lambda) := \sup_{\omega \in \widehat{G}} \# (\Lambda \cap \omega V) < \infty
\end{align}
The quantity $\rel(\Lambda)$ is called the \emph{spreadness} of $\Lambda \subseteq \widehat{G}$.
It suffices to check \eqref{eq:relatively_separated} for a single neighbourhood, cf. the proof of Lemma \ref{lem:neighbourhood} below. 

 The uniform counting estimate provides an upper bound on the spreadness of each annihilator $\Gamma_j^{\perp}$ involving the covolume of $\Gamma_j \subseteq G$. 

\begin{definition}
Let $(\Gamma_j)_{j \in J}$ be a family of closed, co-compact subgroups $\Gamma_j \subseteq G$. The family $(\Gamma_j)_{j \in J}$ is said to satisfy the \emph{uniform counting estimate} if, for every compact set $K \subseteq \widehat{G}$, there exists a constant $C = C(K) > 0$ such that 
\begin{align} \label{eq:counting}
 \sup_{\omega \in \widehat{G}} \sum_{\alpha \in \Gamma_j^{\perp}} \mathds{1}_{K} ( \omega \alpha) \leq C \bigg(1+  \covol(\Gamma_j) \bigg)
\end{align}
for all $j \in J$.
\end{definition}

Counting estimates for lattice points are contained, mainly implicitly, in numerous papers on wavelet theory including \cite{chui2002characterization,calogero2000characterization,hernandez2002unified}. For lattices systems $(\Gamma_j)_{j \in \mathbb{Z}}$ in $\mathbb{R}^d$ whose elements arise as images of a full-rank lattice under integer powers of a non-singular matrix, the uniform counting estimate coincides with the so-called \emph{lattice counting estimate} used in wavelet theory \cite{chui2002characterization,bownik2017wavelets,guo2006some,bownik2011affine,hernandez2002unified}. See also Section \ref{sec:Euclid} below. 

\begin{remark}
In \cite{ron2005generalized}, a family $(\Gamma_j)_{j \in J}$ of full-rank lattices in $\Gamma_j \subseteq \mathbb{R}^d$ is called \emph{round} if there exists a constant $C > 0$ such that, for every $r >0$ and $j \in J$,
\begin{align} \label{eq:round}
\| \sum_{\alpha \in \Gamma_j^{\perp}} \mathds{1}_{B_r(0)} (\cdot + \alpha) \|_{\infty}  \leq 1 + C \covol(\Gamma_j) \mu_{\widehat{\mathbb{R}}^d} (B_r(0)),
\end{align} 
where $B_r (0)$ denotes the Euclidean ball. 
The roundedness condition \eqref{eq:round} differs from the uniform counting estimate \eqref{eq:counting} by the position of the constant $C > 0$. 
Examples in \cite{guo2006some} demonstrate that the uniform counting estimate \eqref{eq:counting}
might be satisfied, while \eqref{eq:round} fails. 
\end{remark}

In verifying the uniform counting estimate, it suffices to do so on a single compact neighbourhood of the identity with non-empty interior.

\begin{lemma} \label{lem:neighbourhood}
Suppose a family $(\Gamma_j)_{j \in J}$ of closed, co-compact subgroups $\Gamma_j \subseteq G$ satisfies \eqref{eq:counting} for a relatively compact neighbourhood of the identity with non-empty interior. Then $(\Gamma_j)_{j \in J}$ satisfies the uniform counting estimate.
\end{lemma}
\begin{proof}
Let $V$ be a relatively compact neighbourhood of $1 \in \widehat{G}$ with interior $V^{\circ} \neq \emptyset$. Suppose that there exists a constant $C_V > 0$ such that
\begin{align*} 
 \sup_{\omega \in \widehat{G}} \sum_{\alpha \in \Gamma_j^{\perp}} \mathds{1}_{V} ( \omega \alpha) \leq C_V \bigg(1+  \covol(\Gamma_j) \bigg) 
\end{align*}
for all $j \in J$. 
Let $K \subseteq \widehat{G}$ be an arbitrary compact set. The collection $\{\omega V^{\circ} \; | \; \omega \in \widehat{G}\}$ forms an open covering for $K \subseteq \widehat{G}$. Hence there exists a finite set $(\omega_i)_{i = 1}^n \subseteq \widehat{G}$ such that $\{\omega_i V^{\circ} \}_{i = 1}^n$ forms a  covering of $K$. Therefore, for all $j \in J$, 
\begin{align*} 
 \sup_{\omega \in \widehat{G}} \sum_{\alpha \in \Gamma_j^{\perp}} \mathds{1}_{K} ( \omega \alpha) &\leq  \sup_{\omega \in \widehat{G}} \sum_{\alpha \in \Gamma_j^{\perp}} \mathds{1}_{\bigcup_{i = 1}^n w_i V} ( \omega \alpha ) \leq  \sum_{i = 1}^n \sup_{\omega \in \widehat{G}}  \sum_{\alpha \in \Gamma_j^{\perp}} \mathds{1}_{V} ( \omega \alpha ) \\
&\leq \sum_{i = 1}^n C_{V} \bigg( 1 + \covol(\Gamma_j) \bigg),
\end{align*}
which gives the desired result. 
\end{proof}

\section{The local integrability condition}
This section considers the relation and the interplay between the local integrability condition,  the Calder\'on integrability condition, the strictly temperateness condition and the uniform counting estimate defined in the previous section. 

The first result provides a simple sufficient condition for the local integrability condition to be satisfied.

\begin{proposition} \label{prop:LIC_sufficient}
Suppose a generalised translation-invariant system $\cup_{j \in J} \{ T_{\gamma} g_{j,p} \}_{\gamma \in \Gamma_j, p \in P_j}$  satisfies the Calder\'on integrability condition \eqref{eq:CalderonLI}, the temperateness condition \eqref{eq:norm-bounded} and the uniform counting estimate \eqref{eq:counting}. Then  $\cup_{j \in J} \{ T_{\gamma} g_{j,p} \}_{\gamma \in \Gamma_j, p \in P_j}$ satisfies the local integrability condition \eqref{eq:LIC}.
\end{proposition}
\begin{proof}
Let $f \in \mathcal{D}_E(G)$ and set $K := \supp \hat{f} \subseteq \widehat{G} \setminus E$. Then applications of Beppo-Levi's theorem and Tonelli-Fubini's theorem yield
\begin{align*}
 &\sum_{j \in J} \frac{1}{\covol (\Gamma_j)} \int_{P_j} \sum_{\alpha
   \in \Gamma_j^{\perp}} \int_{\supp \hat{f}} | 
 \hat{f}(\omega  \alpha)  \hat{g}_{j,p} (\omega)|^2 \;  d\mu_{\widehat{G}} (\omega)d\mu_{P_j}(p) \\
 &\quad \quad \leq \| \hat{f} \|_{\infty}^2 \sum_{j \in J} \frac{1}{\covol (\Gamma_j)} \int_{P_j} \sum_{\alpha
   \in \Gamma_j^{\perp}} \int_{\supp \hat{f}} | 
 \mathds{1}_{K} (\omega \alpha)  \hat{g}_{j,p} (\omega)|^2 \;  d\mu_{\widehat{G}} (\omega) d\mu_{P_j}(p) \\
& \quad \quad =  \| \hat{f} \|_{\infty}^2 \sum_{j \in J} \frac{1}{\covol (\Gamma_j)} \int_{\supp \hat{f}}  \sum_{\alpha
   \in \Gamma_j^{\perp}}  |  \mathds{1}_K (\omega \alpha)  | \int_{P_j} | \hat{g}_{j,p} (\omega)|^2\;  d\mu_{P_j}(p)    d\mu_{\widehat{G}} (\omega).
\end{align*}
Using the uniform counting estimate \eqref{eq:counting}, the calculation can be continued as
\begin{align*}
 &\sum_{j \in J} \frac{1}{\covol (\Gamma_j)} \int_{P_j} \sum_{\alpha
   \in \Gamma_j^{\perp}} \int_{\supp \hat{f}} | 
 \hat{f}(\omega  \alpha)  \hat{g}_{j,p} (\omega)|^2 \;  d\mu_{\widehat{G}} (\omega) d\mu_{P_j} (p)\\
 &\quad \quad \leq \| \hat{f} \|_{\infty}^2 \sum_{j \in J} \frac{1}{\covol (\Gamma_j)} \int_{ \supp \hat{f}} C  \bigg(1 +  \covol(\Gamma_j) \bigg) \int_{P_j}  |\hat{g}_{j,p} (\omega )|^2  \;d\mu_{P_j} (p)  d\mu_{\widehat{G}} (\omega)  \\
 &\quad \quad = C  \| \hat{f} \|_{\infty}^2  \bigg( \int_{\supp \hat{f}} \sum_{j \in J} \frac{1}{\covol(\Gamma_j)} \int_{P_j}  |\hat{g}_{j,p} (\omega )|^2 \; d\mu_{P_j} (p) d\mu_{\widehat{G}} (\omega) \\
&\quad \quad \quad \quad \quad \quad \quad \quad \quad \quad \quad \quad \quad \quad \quad +  \int_{\supp \hat{f}} \sum_{j \in J} \int_{P_j} |\hat{g}_{j,p} (\omega )|^2 \; d\mu_{P_j} (p) d\mu_{\widehat{G}} (\omega) \bigg)\\
&\quad \quad < \infty,
\end{align*}
where the strict inequality follows by the local integrability assumptions \eqref{eq:CalderonLI} and \eqref{eq:norm-bounded}.
\end{proof}

In general, none of the hypotheses in Proposition \ref{prop:LIC_sufficient} can be omitted. See Proposition \ref{prop:counting_compact}, Proposition \ref{prop:counting_discrete} and Remark \ref {rem:construction_wavelet} below.

\begin{lemma} \label{lem:LIC->CalderonLI}
Suppose the generalised translation-invariant system $\cup_{j \in J} \{T_{\gamma} g_{j,p} \}_{\gamma \in \Gamma_j, p \in P_j}$ satisfies the local integrability condition \eqref{eq:LIC} with respect to $E \in \mathcal{E}$. Then $\cup_{j \in J} \{T_{\gamma} g_{j,p} \}_{\gamma \in \Gamma_j, p \in P_j}$ satisfies the Calder\'on integrability condition \eqref{eq:CalderonLI} with respect to $E$.
\end{lemma}
\begin{proof}
 Let $K \subseteq \widehat{G} \setminus E$ be compact. Applications of Beppo Levi's theorem and Tonelli's theorem yield
\begin{align*}
&\int_K \sum_{j \in J} \frac{1}{\covol (\Gamma_j)} \int_{P_j} |\hat{g}_{j,p} (\omega)|^2 \; d\mu_{P_j} (p) d\mu_{\widehat{G}} (\omega) \\
& \quad \quad \quad \quad \quad =  \sum_{j \in J} \frac{1}{\covol (\Gamma_j)}  \int_{P_j} \int_K |\hat{g}_{j,p} (\omega)|^2 \;  d\mu_{\widehat{G}} (\omega) d\mu_{P_j} (p)  \\
 & \quad \quad \quad \quad \quad \leq \sum_{j \in J} \frac{1}{\covol (\Gamma_j)} \int_{P_j} \sum_{\alpha \in \Gamma_j^{\perp}} \int_{K \cap \alpha^{-1} K} | \hat{g}_{j,p} (\omega)|^2 \;  d\mu_{\widehat{G}} (\omega) d\mu_{P_j} (p) < \infty,
\end{align*}
as required.
\end{proof}

The Calder\'{o}n integrability condition is automatically satisfied for any generalised translation-invariant system $\cup_{j \in J} \{T_{\gamma} g_{j,p} \}_{\gamma \in \Gamma_j, p \in P_j}$ forming a Bessel system for $L^2 (G)$ or satisfying the so-called CC-condition \cite{lemvig2017sufficient}. The Calder\'{o}n sum of such systems is namely uniformly bounded by the upper frame bound, cf. \cite[Proposition 3.3]{jakobsen2016reproducing}. The frame property is, however, not related to the local integrability condition. For example, a generalised shift-invariant system that forms an orthonormal basis for $L^2 (\mathbb{R})$, but that fails the local integrability condition is contained in \cite[Example 3.2]{bownik2004spectral}. See also \cite[Proposition 4.6]{kutyniok2006theory} for a similar example in $\ell^2 (\mathbb{Z})$. 

The following simple observation is an adaption of \cite[Proposition 3.2]{christensen2017explicit}. 

\begin{lemma}
Suppose $\cup_{j \in J} \{T_{\gamma} g_{j,p} \}_{p \in P_j, \gamma \in \Gamma_j}$ satisfies the Calder\'on integrability condition \eqref{eq:CalderonLI} with respect to $E \in \mathcal{E}$. Then the generating functions $\cup_{j \in J} \{g_{j,p} \}_{p \in P_j}$ satisfy the temperateness condition \eqref{eq:norm-bounded}, with respect to $E$, if there exists a constant $N > 0$ such that
\begin{align} \label{eq:simplerLI}
\sum_{\{j \in J \; | \; \covol(\Gamma_j) > N\}} \int_{P_j} |\hat{g}_{j,p}(\cdot) |^2 \; d\mu_{P_j} (p) \in L^1_{\loc} (\widehat{G} \setminus E). 
\end{align}
\end{lemma}
\begin{proof}
Suppose the Calder\'on integrability condition \eqref{eq:CalderonLI} is satisfied for some $E \in \mathcal{E}$ and that \eqref{eq:simplerLI} holds for an $N > 0$. Then
\begin{align*}
\sum_{j \in J} \int_{P_j} |\hat{g}_{j,p}(\omega) |^2 \; d\mu_{P_j} (p) &\leq N \sum_{\{j \in J \; | \; \covol(\Gamma_j) \leq N\}} \frac{1}{\covol(\Gamma_j)} \int_{P_j} |\hat{g}_{j,p}(\omega) |^2 \; d\mu_{P_j} (p) \\
& \quad \quad \quad + \sum_{\{j \in J \; | \; \covol(\Gamma_j) > N\}} \int_{P_j} |\hat{g}_{j,p}(\omega) |^2 \; d\mu_{P_j} (p)
\end{align*}
for all $\omega \in \widehat{G} \setminus E$. Integration over a compact set $K \subseteq \widehat{G} \setminus E$ therefore gives the desired result. 
\end{proof}

Contrary to the Calder\'on integrability condition \eqref{eq:CalderonLI}, the temperateness condition \eqref{eq:norm-bounded} is not a necessary condition for the local integrability condition \eqref{eq:LIC} as shown by the next example.

\begin{example} \label{ex:mainexample}
Let $G = \mathbb{R}^2$, let $J = \mathbb{N}$ and fix some $0 < a < \frac{1}{10}$. For each $j \in \mathbb{N}$, let
\[
C_j := \begin{pmatrix}
a & j \\
-a & j 
\end{pmatrix} \in \mathrm{GL}(2, \mathbb{R}),
\]
and define the full-rank lattices $\Gamma_j = C_j \mathbb{Z}^2$. 

For $n \in \mathbb{N}$, let
\[
I_n := [-1 + 2^{-n}, -1 + 2^{-(n-1)})^2
\]
and
\[
I_0 := [-1,1]^2 \setminus \bigcup_{n \in \mathbb{N}} I_n.
\]
Observe that $[-1,1]^2 = \bigcup_{n \in \mathbb{N}_0} I_n$ with the union being disjoint. 

It is claimed that the following counting estimate holds
\begin{align} \label{eq:supIn}
\sup_{\omega \in I_n} \sum_{k \in \mathbb{Z}^2} \mathds{1}_{[-1,1]^2} (\omega + (C_j^T)^{-1} k) \leq 4j 2^{-(n-1)} + 1
\end{align}
for all $j \in \mathbb{N}$ and $n \in \mathbb{N}_0$. In order to see this, let $j \in \mathbb{N}$ be fixed. Consider the cases $n = 0$ and $n \in \mathbb{N}$. For $n = 0$, clearly $[-1,1]^2 - \omega \subseteq [-2, 2]^2$ for any $w \in I_0$. Therefore
\[
\# \bigg( \Gamma_j^{\perp} \cap( [-1,1]^2 - \omega) \bigg) \leq \# \bigg( \Gamma_j^{\perp} \cap [-2,2]^2 \bigg)
\]
for all $\omega \in I_0$. Note that any $\alpha \in \Gamma_j^{\perp} = (C_j^T)^{-1} \mathbb{Z}^2$ is of the form
\[
\alpha = 
\begin{pmatrix}
\frac{m_1}{2a} + \frac{m_2}{2j} \\
\frac{m_1}{2a} - \frac{m_2}{2j}
\end{pmatrix} \in \mathbb{R}^2
\]
for some $(m_1, m_2) \in \mathbb{Z}^2$. Hence $\alpha \in [-2, 2]^2$ if, and only if,
\begin{align} \label{eq:condn0}
\frac{m_1}{2a} + \frac{m_2}{2j} \in [-2,2] \quad \text{and} \quad
\frac{m_1}{2a} - \frac{m_2}{2j} \in [-2,2].
\end{align}
Since $a < \frac{1}{10}$ by assumption, it follows that necessarily $m_1 = 0$ whenever $\alpha \in [-2,2]^2$. Thus condition \eqref{eq:condn0} becomes $m_1 = 0$ and $m_2  \in [-4j,4j]$. But the number of points $m_2 \in \mathbb{Z}$ that are contained in $[-4j,4j]$ does not exceed $8j + 1$, whence \eqref{eq:supIn} follows for $n = 0$. To show \eqref{eq:supIn} for the case $n \in \mathbb{N}$, fix an arbitrary $n \in \mathbb{N}$. Observe that $[-1,1]^2 - \omega \subseteq [-2^{-(n-1)}, 2 - 2^{-n}]^2$ for any $\omega \in I_n$. Since $\alpha \in [-2^{-(n-1)}, 2 - 2^{-n}]^2$ is equivalent to
\begin{align} \label{eq:condn}
\frac{m_1}{2a} + \frac{m_2}{2j} \in [-2^{-(n-1)}, 2 - 2^{-n}] \quad \text{and} \quad
\frac{m_1}{2a} - \frac{m_2}{2j} \in [-2^{-(n-1)}, 2 - 2^{-n}],
\end{align}
again necessarily $m_1 = 0$ provided that $\alpha \in [-2^{-(n-1)}, 2 - 2^{-n}]^2$. Therefore, condition \eqref{eq:condn} reads $m_1 = 0$ and
\[
\frac{n_2}{2j} \in [-2^{-(n-1)}, 2 - 2^{-n}] \cap [-2 + 2^{-n}, 2^{-(n-1)}] = [-2^{-(n-1)}, 2^{-(n-1)}].
\]
Since the number of points $m_2 \in \mathbb{Z}$ satisfying $|m_2| \leq 2j 2^{-(n-1)}$ is at most $4j2^{-(n-1)} + 1$, the claim \eqref{eq:supIn} follows.

Define $\eta : [-1,1]^2 \to \mathbb{R}$ by $\eta (\omega) = \mu_{\widehat{\mathbb{R}}^2} (I_n)^{-1}$ for $\omega \in I_n$. Since $(I_n)_{n \in \mathbb{N}_0}$ forms a partition of $[-1,1]^2$, the function $\eta$ is well-defined. Using $\eta$, define next the Borel sets
\[
\Omega_j := \{ j-1 \leq \eta < j \} = \{\omega \in [-1,1]^2 \; | \; j-1 \leq \eta (\omega) < j \}
\]
for $j \in \mathbb{N}$. Then $[-1,1]^2 = \bigcup_{j \in J} \Omega_j$, where the union is disjoint. For each $j \in \mathbb{N}$, define  $g_j : \mathbb{R}^2 \to \mathbb{C}$ in the Fourier domain as $\hat{g}_j = \eta^{1/2} \cdot \mathds{1}_{\Omega_j}$.
Then $g_j \in L^2 (\mathbb{R}^2)$ since 
\[
\int_{\widehat{\mathbb{R}}^2} | \hat{g}_j (\omega) |^2 \; d\mu_{\widehat{\mathbb{R}}^2} (\omega) = \int_{\Omega_j} |\eta (\omega)| \; d\mu_{\widehat{\mathbb{R}}^2} (\omega) \leq j \cdot d\mu_{\widehat{\mathbb{R}}^2} (\Omega_j) < \infty,
\]
where it is used that $\Omega_j \subseteq [-1,1]^2$ for each $j \in \mathbb{N}$.

In order to show that $\cup_{j \in J} \{T_{\gamma} g_j \}_{\gamma \in \Gamma_j}$ satisfies the local integrability condition, it suffices to take $f \in \mathcal{D}_{\emptyset} (\mathbb{R}^2)$ with $\supp \hat{f} \subseteq [-1,1]^2$ since $\supp\hat{g}_{j} \subseteq [-1,1]^2$ for any $j \in \mathbb{N}$ by construction. Using the estimate \eqref{eq:supIn}, it follows that
\begin{align*}
&\sum_{j \in J} \frac{1}{\covol(\Gamma_j)} \int_{\supp \hat{f}} \sum_{\alpha \in \Gamma_j^{\perp}} |\hat{f} (\omega + \alpha)  \hat{g}_j (\omega)|^2 \; d\mu_{\widehat{\mathbb{R}}^2} (\omega)  \\
& \quad \quad \quad \leq \| \hat{f} \|_{\infty}^2 \sum_{j \in \mathbb{N}} \frac{1}{2aj} \int_{[-1,1]^2} \sum_{\alpha \in \Gamma_j^{\perp}} \mathds{1}_{[-1,1]^2} (\omega + \alpha) | \hat{g}_j (\omega)|^2 \; d\mu_{\widehat{\mathbb{R}}^2} (\omega) \\
&\quad \quad \quad = \| \hat{f} \|_{\infty}^2 \sum_{j \in \mathbb{N}} \frac{1}{2a j} \sum_{n \in \mathbb{N}_0} \int_{I_n} \sum_{\alpha \in \Gamma_j^{\perp}} \mathds{1}_{[-1,1]^2} (\omega + \alpha) | \hat{g}_j (\omega)|^2 \; d\mu_{\widehat{\mathbb{R}}^2} (\omega) \\
&\quad \quad \quad \leq \| \hat{f} \|_{\infty}^2 \sum_{j \in \mathbb{N}} \frac{1}{2a j} \sum_{n \in \mathbb{N}_0} \bigg( \int_{I_n}  4j 2^{-(n-1)}  | \hat{g}_j (\omega)|^2 \; d\mu_{\widehat{\mathbb{R}}^2} +  \int_{I_n} | \hat{g}_j (\omega)|^2 \; d\mu_{\widehat{\mathbb{R}}^2} (\omega) \bigg) \\
&\quad \quad \quad = \| \hat{f} \|_{\infty}^2 \sum_{n \in \mathbb{N}_0} \bigg( \frac{2^{2-n}}{a} \int_{I_n} \sum_{j \in \mathbb{N}} | \hat{g}_j (\omega)|^2 \; d\mu_{\widehat{\mathbb{R}}^2} (\omega) +  \int_{I_n} \sum_{j \in \mathbb{N}}\frac{1}{2aj} | \hat{g}_j (\omega)|^2 \; d\mu_{\widehat{\mathbb{R}}^2} (\omega)\bigg).
\end{align*}
It remains therefore to show that the two summands in the last equality above are finite. For the first summand, observe that $\eta (\omega) = \sum_{j \in \mathbb{N}} |\hat{g}_j (\omega)|^2$ by construction. Therefore
\begin{align*}
\sum_{n \in \mathbb{N}_0} \frac{2^{2-n}}{a} \int_{I_n} \sum_{j \in \mathbb{N}} | \hat{g}_j (\omega)|^2 \; d\mu_{\widehat{\mathbb{R}}^2} (\omega)  = \sum_{n \in \mathbb{N}_0} \frac{2^{2-n}}{a} \int_{I_n} \mu_{\widehat{\mathbb{R}}^2} (I_n)^{-1} \; d\mu_{\widehat{\mathbb{R}}^2} (\omega) =  \sum_{n \in \mathbb{N}_0} \frac{2^{2-n}}{a} = \frac{8}{a}.
\end{align*}
For the second summand, a direct calculation gives
\begin{align*}
 \int_{[-1,1]^2} \sum_{j \in \mathbb{N}}\frac{1}{2aj} | \hat{g}_j (\omega)|^2 \; d\mu_{\widehat{\mathbb{R}}^2} (\omega) &= \sum_{j \in \mathbb{N}} \frac{1}{2aj} \int_{\Omega_j} \eta (\omega) \; d\mu_{\widehat{\mathbb{R}}^2} (\omega) \leq \frac{1}{2a} \sum_{j \in \mathbb{N}}  \int_{\Omega_j} \; d\mu_{\widehat{\mathbb{R}}^2} (\omega) \\
&= (2a)^{-1} \mu_{\widehat{\mathbb{R}}^2} ([-1,1]^2).
\end{align*}
Thus the local integrability condition \eqref{eq:LIC} is satisfied. On the other hand,
\[
\int_{[-1,1]^2} \sum_{j \in \mathbb{N}} |\hat{g}_j (\omega)|^2 \; d\mu_{\widehat{\mathbb{R}}^2} (\omega) = \sum_{n \in \mathbb{N}_0} \int_{I_n} \eta(\omega) \; d\mu_{\widehat{\mathbb{R}}^2} (\omega) = \sum_{n \in \mathbb{N}_0} 1 = \infty,
\]
showing that \eqref{eq:norm-bounded} is not satisfied. 
\end{example}

The next result asserts that under an additional assumption on the translation subgroups $(\Gamma_j)_{j \in J}$, the temperateness condition \eqref{eq:simplerLI} is necessary for the local integrability condition \eqref{eq:LIC} to be satisfied. 

\begin{lemma} \label{lem:lowerCE}
Let $(\Gamma_j)_{j \in J}$ be a family of closed, co-compact subgroups $\Gamma_j \subseteq G$ and let $E \in \mathcal{E}$. Suppose that for any compact set $K \subseteq \widehat{G} \setminus E$ there exists a compact superset $Q \subseteq \widehat{G} \setminus E$ and a constant $C = C(K) >0$ such that
\begin{align} \label{eq:lowerCE}
\essinf_{\omega \in Q} \sum_{\alpha \in \Gamma_j^{\perp}} 1_{Q} (\omega \alpha) \geq C \covol(\Gamma_j) 
\end{align}
 for all $j \in J$.  Moreover,  suppose that $\cup_{j \in J} \{T_{\gamma} g_{j,p} \}_{\gamma \in \Gamma_j, p \in P_j}$ satisfies the local integrability condition \eqref{eq:LIC} with respect to $E$. Then the generating functions $\cup_{j \in J} \{g_{j,p} \}_{p \in P_j}$ are temperate \eqref{eq:norm-bounded} with respect to $E$.
 \end{lemma}
\begin{proof}
Fix a compact set $K \subseteq \widehat{G} \setminus E$ and let $Q \subseteq \widehat{G} \setminus E$ be any compact set satisfying $K \subseteq Q$ and \eqref{eq:lowerCE}. Then
\begin{align*}
\infty &>  \sum_{j \in J} \frac{1}{\covol (\Gamma_j)} \int_{P_j} \sum_{\alpha \in \Gamma_j^{\perp}} \int_{Q \cap \alpha^{-1} Q} | \hat{g}_{j,p} (\omega)|^2 \;  d\mu_{\widehat{G}} (\omega) d\mu_{P_j} (p) \\
&= \sum_{j \in J} \frac{1}{\covol (\Gamma_j)} \int_{P_j}  \int_{Q} \sum_{\alpha \in \Gamma_j^{\perp}} \mathds{1}_{  \alpha^{-1} Q} (\omega) | \hat{g}_{j,p} (\omega)|^2 \;  d\mu_{\widehat{G}} (\omega) d\mu_{P_j} (p) \\
&\geq C \sum_{j \in J}  \int_{P_j}  \int_{Q}   | \hat{g}_{j,p} (\omega)|^2 \;  d\mu_{\widehat{G}} (\omega) d\mu_{P_j} (p) \geq C  \sum_{j \in J}  \int_{P_j}  \int_{K}   | \hat{g}_{j,p} (\omega)|^2 \;  d\mu_{\widehat{G}} (\omega) d\mu_{P_j} (p), 
\end{align*}
which shows the result. 
\end{proof}

\begin{remark}
Condition \eqref{eq:lowerCE} provides a lower bound on the number of lattice points that intersect the translates of a compact set. For the case $G = \mathbb{R}^d$, a lower bound on the number of lattice points that intersect a fixed $0$-symmetric convex body is provided by Minkowski's convex body theorem, e.g., cf. \cite[Theorem 3.28]{tao2006additive}. See also the volume packing lemma \cite[Lemma 3.24]{tao2006additive}. A version for non-symmetric convex bodies can be found as \cite[Theorem 2.5]{lagarias1991bounds}.
\end{remark}

The section will be concluded with an example of a system satisfying the local integrability condition \eqref{eq:LIC}, but failing the uniform counting estimate \eqref{eq:counting}.

\begin{example} \label{ex:failingUCEbutLIC}
Let $G = \mathbb{R}^2$ and let $J = \mathbb{N}$. For each $j \in \mathbb{N}$, define 
\[ C_j := 
\begin{pmatrix}
j^{-1} & 0 \\
0 & j
\end{pmatrix}
\in \mathrm{GL}(2, \mathbb{R}).
\]
Set $\Gamma_j := C_j \mathbb{Z}^2$. Then $\covol (\Gamma_j) = |\det C_j |= 1$ for all $j \in \mathbb{N}$. To show that the uniform counting estimate \eqref{eq:counting} fails for $(\Gamma_j)_{j \in J}$, take $K := [0,1]^2$. Then, for fixed $j \in \mathbb{N}$, 
\begin{align*}
\sup_{\omega \in \mathbb{R}^2} \sum_{\alpha \in \Gamma_j^{\perp}} \mathds{1}_K (\omega + \alpha) \geq \sum_{\alpha \in \Gamma_j^{\perp}} \mathds{1}_K (\alpha) = \sum_{m \in \mathbb{Z}} \sum_{n \in \mathbb{Z}} \mathds{1}_K (jm, j^{-1} n).
\end{align*}
Now, since
\[ \# \bigg\{ (m,n) \in \mathbb{Z}^2 \; \bigg| \; (jm, j^{-1}n) \in [0,1]^2 \bigg\} \geq \# \bigg\{ n \in \mathbb{Z} \; \bigg| \; j^{-1} n \in [0,1] \bigg\} = j + 1, \]
it follows that
\begin{align*}
\sup_{\omega \in \mathbb{R}^2} \sum_{\alpha \in \Gamma_j^{\perp}} \mathds{1}_K (\omega + \alpha) \geq j + 1,
\end{align*}
which shows that the uniform counting estimate \eqref{eq:counting} fails. 

Consider next the index sets $P_j = \mathbb{Z}^2$ for all $j \in J$, and let $N \in \mathbb{R}$ be such that $N > 1$. Define $g_{j,p}  \in  L^2 (\mathbb{R}^2)$ by $\hat{g}_{j,p} = N^{-j} T_p \mathds{1}_{[0, 1]^2}$ for each $j \in J$ and $p \in P_j$.
Then
\[
\sum_{j \in \mathbb{N}} \sum_{p \in P} |\hat{g}_{j,p} (\omega)|^2 = \sum_{j \in \mathbb{N}} N^{-j} \sum_{p \in \mathbb{Z}^2}  T_p \mathds{1}_{[0, 1]^2} (\omega) =  \sum_{j \in \mathbb{N}} N^{-j} \mathds{1}_{\mathbb{R}^2} (\omega) =  \frac{1}{N - 1}
\]
for all $\omega \in \mathbb{R}^2$, whence conditions \eqref{eq:CalderonLI} and \eqref{eq:norm-bounded} are satisfied. 
In order to show that $\cup_{j \in J} \{T_{\gamma} g_{j,p} \}_{\gamma \in \Gamma_j, p \in P_j}$ also satisfies the LIC,  let $f \in \mathcal{D}_{\emptyset} (\mathbb{R}^2)$ be arbitrary and choose $r>0$ such that $\supp \hat{f} \subseteq [-r, r]^2$. Then
\begin{align*}
&\sum_{j \in \mathbb{N}} \sum_{p \in \mathbb{Z}^2} \sum_{\alpha \in \Gamma_j^{\perp}} \int_{\supp \hat{f}} |\hat{f}(\omega + \alpha) \hat{g}_{j,p} (\omega)|^2 \; d\mu_{\widehat{\mathbb{R}}^2} (\omega) \\
& \quad \quad \quad \quad \leq \| \hat{f} \|_{\infty}^2 \sum_{j \in \mathbb{N}} \sum_{p \in \mathbb{Z}^2} \sum_{\alpha \in \Gamma_j^{\perp}} \int_{[-r, r]^2} |\mathds{1}_{[-r,r]^2} (\omega+ \alpha) || \hat{g}_{j,p} (\omega)|^2 \; d\mu_{\widehat{\mathbb{R}}^2} (\omega) \\
& \quad \quad \quad \quad = \| \hat{f} \|_{\infty}^2 \sum_{j \in \mathbb{N}}  \sum_{p \in \mathbb{Z}^2}  \int_{[-r, r]^2} \sum_{k \in \mathbb{Z}^2}  |\mathds{1}_{[-r,r]^2} (\omega + (C_j^T)^{-1} k) || \hat{g}_{j,p} (\omega)|^2 \; d\mu_{\widehat{\mathbb{R}}^2} (\omega) \numberthis \label{eq:LIC_estimate}
\end{align*}
where interchanging the sum and integral is justified by Beppo-Levi's theorem. Next, for any fixed $\omega \in [-r, r]^2$, the innermost series in \eqref{eq:LIC_estimate} can be estimated as
\begin{align*}
 \sum_{k \in \mathbb{Z}^2}  \mathds{1}_{[-r,r]^2} (\omega + C_j^{-1} k)  &= 
\# \bigg\{ (m,n) \in \mathbb{Z}^2 \; \bigg| \; \omega + (jm, j^{-1} n) \in [-r, r]^2 \bigg\} \\
 &\leq \# \bigg\{ (m,n) \in \mathbb{Z}^2 \; \bigg| \;  (jm, j^{-1} n) \in [-2r, 2r]^2 \bigg\} \\
&= \# \bigg\{ m \in \mathbb{Z} \; \bigg| \; j m \in [-2r, 2r] \bigg\} \cdot \# \bigg \{ n \in \mathbb{Z} \; \bigg| \; j^{-1} n \in [-2r, 2r] \bigg\} \\
&\leq  (4j^{-1} r + 1) (4r j + 1).
\end{align*}
Using this uniform bound and the estimate \eqref{eq:LIC_estimate}, it follows directly that
\begin{align*}
&\sum_{j \in \mathbb{N}} \sum_{p \in \mathbb{Z}^2} \sum_{\alpha \in \Gamma_j^{\perp}} \int_{\supp \hat{f}} |\hat{f}(\omega + \alpha) \hat{g}_{j,p} (\omega)|^2 \; d\mu_{\widehat{\mathbb{R}}^2} (\omega) \\
&\leq  \| \hat{f} \|_{\infty}^2 \sum_{j \in \mathbb{N}}  (4j^{-1} r + 1) (4r j + 1) \int_{[-r, r]^2} \sum_{p \in \mathbb{Z}^2}  | \hat{g}_{j,p} (\omega)|^2 \; d\mu_{\widehat{\mathbb{R}}^2} (\omega) \\
&\leq  \| \hat{f} \|_{\infty}^2 \mu_{\mathbb{R}^2} ([-r,r]^2) \sum_{j \in \mathbb{N}}  \frac{(4j^{-1} r + 1) (4r j + 1)}{N^j} < \infty.
\end{align*}
Since $f \in \mathcal{D}_{\emptyset} (\mathbb{R}^2)$ was taken arbitrary, this shows \eqref{eq:LIC}. 
\end{example}

\section{The local integrability condition on special groups}
This section considers the hypotheses of Proposition \ref{prop:LIC_sufficient} on some specific groups. The first structure theorem for locally compact Abelian groups asserts that $G \cong  G_0 \times \mathbb{R}^d$ for some $d \in \mathbb{N}_0$ and a locally compact Abelian group $G_0$ containing a compact open subgroup. The next two subsections are devoted to the local integrability condition on locally compact Abelian groups possessing a compact open subgroup and the Euclidian space, respectively.

\subsection{Locally compact Abelian groups containing a compact open subgroup}
In this subsection the local integrability condition is considered for the case in which the underlying locally compact Abelian group $G$ possesses a compact open subgroup. Examples of such groups include all discrete groups, compact groups and totally disconnected groups. For a wavelet theory on locally compact Abelian groups containing a compact open subgroup, the interested reader is referred to \cite{benedetto2004wavelet}. 

Let $H \subseteq G$ be a compact open subgroup of $G$. The subgroup $H$ being open is equivalent to the quotient $G / H$ being discrete. Hence $H^{\perp} \cong \widehat{G / H}$ is a compact subgroup of $\widehat{G}$. Moreover, since $\widehat{G} / H^{\perp} \cong \widehat{H}$ is discrete, it follows that $H^{\perp}$ is also open. Thus $H^{\perp}$ is compact and open. 

The Haar measures in this section will be chosen as follows. For the normalisation, take an arbitrary, but fixed, compact open subgroup $H \subseteq G$. If $G$ is compact or discrete, the subgroup will simply be taken to be $H = G$ or $H = \{0\}$, respectively. Relative to this fixed subgroup $H \subseteq G$, the Haar measure $\mu_G$ on $G$ will be normalised such that $\mu_G (H) = 1$. The Plancherel measure $\mu_{\widehat{G}}$ on $\widehat{G}$ is then the unique Haar measure $\mu_{\widehat{G}}$ with $\mu_{\widehat{G}} (H^{\perp}) = 1$. The Haar measures on $H$ and $H^{\perp}$ will be taken to be the (Haar) measures that are induced by the given measures $\mu_{G}$ and $\mu_{\widehat{G}}$, respectively.  The quotients $G / H$ and $\widehat{G} / H^{\perp}$ will both be equipped with the counting measure. With these conventions, the chosen Haar measures are canonically related and satisfy Plancherel's theorem, cf. \cite[Section 31.1]{hewitt1970abstract}. 

By Lemma \ref{lem:neighbourhood}, it suffices to verify the uniform counting estimate on a compact neighbourhood of the identity. Hence, the uniform counting estimate \eqref{eq:counting} is satisfied whenever it is satisfied for the compact open subgroup $H^{\perp} \subseteq \widehat{G}$. This gives rise to the following two results. 

\begin{proposition} \label{prop:counting_compact}
Let $G$ be a compact Abelian group and let $(\Gamma_j)_{j \in J}$ be a family of closed subgroups $\Gamma_j \subseteq G$. Then $(\Gamma_j)_{j \in J}$ satisfies the uniform counting estimate \eqref{eq:counting}. Moreover, the local integrability condition \eqref{eq:LIC} is equivalent to Calder\'on integrability condition \eqref{eq:CalderonLI}.
\end{proposition}
\begin{proof}
Let $G$ be a compact Abelian group. Then its Pontryagin dual $\widehat{G}$ is discrete, and $K := \{1\}$ is a compact open subgroup in $\widehat{G}$. For any $\omega \in \widehat{G}$, it follows that $\# ( \Gamma_j^{\perp} \cap (K - \omega) ) \leq 1$ and thus the uniform counting estimate is satisfied. 

In case $\Gamma_j \subseteq G$ forms a uniform lattice, it is by assumption equipped with the counting measure. It follows that $\covol(\Gamma_j) \leq 1 = \mu_{G} (G) $. If $\Gamma_j \subseteq G$ is non-discrete, the Haar measure on $\Gamma_j$ is by assumption to be the probability measure, which directly gives $\covol(\Gamma_j) = 1$. Thus $\covol(\Gamma_j) \leq 1$ for any $j \in J$. Therefore, in both cases, condition \eqref{eq:CalderonLI} is stronger than condition \eqref{eq:norm-bounded}. By Proposition \ref{prop:LIC_sufficient} and Lemma \ref{lem:LIC->CalderonLI}, it follows that the local integrability condition \eqref{eq:LIC}  is equivalent to the Calder\'on integrability condition \eqref{eq:CalderonLI}.
\end{proof}

\begin{proposition} \label{prop:counting_discrete}
Let $G$ be a discrete Abelian group and let $(\Gamma_j)_{j \in J}$ be a family of closed, co-compact subgroups $\Gamma_j \subseteq G$. Then $(\Gamma_j)_{j \in J}$ satisfies the uniform counting estimate \eqref{eq:counting}. Moreover, the local integrability condition \eqref{eq:LIC}  is equivalent to the temperateness condition \eqref{eq:norm-bounded}.
\end{proposition}
\begin{proof}
Since the group $G$ is assumed to be discrete, its dual group $\widehat{G}$ is compact. Set $K = \widehat{G}$. Then $\# ( \Gamma_j^{\perp} \cap (K - \omega) ) = \# \Gamma_j^{\perp} = \covol(\Gamma_j)$ for any $\omega \in \widehat{G}$ and $j \in J$, which shows that the uniform counting estimate is satisfied. Using that $\covol(\Gamma_j) = \# \Gamma_j^{\perp} \geq 1$ for all $j \in J$, it follows that the Calder\'on integrability condition \eqref{eq:CalderonLI} is weaker than \eqref{eq:norm-bounded}. Combining Proposition \ref{prop:LIC_sufficient} and Lemma \ref{lem:LIC->CalderonLI} therefore yields that the local integrability condition \eqref{eq:LIC}  coincides with the temperateness condition \eqref{eq:norm-bounded}.
\end{proof}

For a general locally compact Abelian group possessing a compact open subgroup, the uniform counting estimate is a non-trivial condition:

\begin{example} \label{ex:compact_open}
Let $G = \mathbb{Z} \times \mathbb{T}$ and consider the open compact subgroup $H = \{0\} \times \mathbb{T} \subseteq G$. For each $j \in \mathbb{N}$, define $\Gamma_j = j \mathbb{Z} \times \mathbb{T}_j$, where $\mathbb{T}_j := \{e^{2\pi i k / j} \; | \; k = \{0, ..., j-1\} \}$. Then $H^{\perp} = \mathbb{T} \times \{0\}$ and $\Gamma_j^{\perp} = \mathbb{T}_j \times j \mathbb{Z}$. Then
\[
\sum_{\alpha \in \Gamma^{\perp}_j} 1_{H^{\perp}} (\alpha) = \sum_{z \in \mathbb{T}_j} 1_{\mathbb{T}} (z) = j,
\]
while $\covol(\Gamma_j) = 1$ for all $j \in \mathbb{N}$.  
\end{example}

For the special case in which all translation subgroups $(\Gamma_j)_{j \in J}$ of the system $\cup_{j \in J} \{T_{\gamma} g_j \}_{\gamma \in \Gamma_j}$ are discrete, a complete characterisation of the local integrability condition, with respect to $E = \emptyset$, on groups possessing a compact connected component was given by Kutyniok \& Labate \cite{kutyniok2006theory}. Theorem \ref{thm:LIC_compactopen}  below extends this result to systems with possibly non-discrete translation subgroups $(\Gamma_j)_{j \in J}$ in a locally compact Abelian group containing a compact open subgroup. This extension is of interest since the only uniform lattice in a locally compact Abelian group might be the trivial group or the group might not contain any uniform lattice at all. For example, in a torsion-free compact Abelian group the only finite subgroup is the trivial subgroup $\{0\}$ and hence it is the only uniform lattice in such a group. On the other hand, for a prime $p \in \mathbb{N}$, the direct product $G = \mathbb{Q}_p \times \mathbb{Z}_p$ of the $p$-adic numbers $\mathbb{Q}_p$ and the $p$-adic integers $\mathbb{Z}_p$ does not contain any uniform lattice, but it contains a countably infinite number of distinct closed, co-compact subgroups. 

\begin{theorem} \label{thm:LIC_compactopen}
Let $G$ be a locally compact Abelian group containing a compact open subgroup $H \subseteq G$. Let $\cup_{j \in J} \{T_{\gamma} g_{j,p} \}_{\gamma \in \Gamma_j, p \in P_j}$ be a generalised translation-invariant system in $L^2 (G)$. Then the following assertions are equivalent:
\begin{enumerate}[(i)]
\item The system $\cup_{j \in J} \{T_{\gamma} g_{j,p} \}_{\gamma \in \Gamma_j, p \in P_j}$ satisfies the local integrability condition \eqref{eq:LIC} with respect to $E = \emptyset$. 
\item The following integrability condition holds:
\begin{align} \label{eq:LIC_compactopen}
\sum_{j \in J} \frac{1}{\covolh (\Gamma_j)} \int_{P_j}  |\hat{g}_{j,p} (\cdot)|^2 \; d\mu_{P_j} (p) \in L^1_{\loc} (\widehat{G}),
\end{align}
where $\covolh (\Gamma_j) := \mu_{H\Gamma_j / \Gamma_j} (H\Gamma_j / \Gamma_j)$. 
\end{enumerate}
\end{theorem}

An essential part of the proof in \cite{kutyniok2006theory} relies crucially on the assumption that the translation subgroups  are discrete. Therefore, a proof of Theorem \ref{thm:LIC_compactopen} will be supplied below. The proof will make use of the following two lemmata.

\begin{lemma} \label{lem:covolumes}
Let $G$ be a locally compact group and let $\Gamma \subseteq G$ and $\Lambda \subseteq G$ be closed subgroups satisfying $\Gamma \subseteq \Lambda \subseteq G$. Suppose that $\Gamma \subseteq G$ is co-compact. Then $\Lambda$ is co-compact, and
\[
\covol(\Gamma) = \mu_{\Lambda/\Gamma}(\Lambda/\Gamma) \covol(\Lambda).
\]
\end{lemma}
\begin{proof}
Let $\Gamma \subseteq \Lambda \subseteq G$ be closed subgroups with given Haar measure $\mu_{\Gamma}$ and $\mu_{\Lambda}$, respectively.  Take a Bruhat function $\varphi \in C_c (G)$ for $\Gamma$, i.e., a function $\varphi \in C_c (G)$ satisfying
\[
 \int_{\Gamma} \varphi (x+\gamma) \; d\mu_{\Gamma} (\gamma) = 1
\]
for all $x \in G$. The existence of such a function is provided by \cite[Theorem 8.1.20]{reiter2000classical}. Define  $\varphi_{\Gamma} : G/\Gamma \to \mathbb{C}$ by $\varphi_{\Gamma} (\dot{x}) := \int_{\Gamma}  \varphi (x+\gamma) \; d\mu_{\Gamma} (\gamma)$, where $\dot{x} := x + \Gamma$ for $x \in G$. Then $\varphi_{\Gamma} \in L^1 (G / \Gamma)$ since $\varphi \in C_c (G) \subseteq L^1 (G)$. Applications of Weil's integral formula yield
\begin{align*}
\covol(\Gamma) &= \int_{G / \Gamma} 1 \; d\mu_{G/ \Gamma} (\dot{x}) = \int_{G / \Gamma} \int_{\Gamma} \varphi (x+\gamma) \; d\mu_{\Gamma} (\gamma) d\mu_{G/ \Gamma} (\dot{x}) = \int_G \varphi (x) \; d\mu_{G} (x) \\
&= \int_{G / \Lambda} \int_{\Lambda} \varphi (x + \lambda) \; d\mu_{\Lambda} (\lambda) d \mu_{G/\Lambda} (\dot{x}) = \int_{G/\Lambda} \int_{\Lambda / \Gamma} \int_{\Gamma} \varphi (x + \eta + \gamma) \; d\mu_{\Gamma} (\gamma) d\mu_{\Lambda / \Gamma} (\eta) d\mu_{G / \Lambda} (\dot{x}) \\
&= \int_{G / \Lambda} \int_{\Lambda / \Gamma} \varphi_{\Gamma} (x + \eta) \; d\mu_{\Lambda / \Gamma} (\eta) d\mu_{G / \Lambda} (\dot{x}) = \int_{G / \Lambda} \int_{\Lambda / \Gamma} 1 \; d\mu_{\Lambda / \Gamma} (\eta) d\mu_{G / \Lambda} (\dot{x}) \\
&= \mu_{\Lambda / \Gamma} (\Lambda / \Gamma) \mu_{G / \Lambda} (G / \Lambda).
\end{align*}
Since $\Gamma \subseteq G$ is assumed to be co-compact, it follows directly that also $\covol (\Lambda) < \infty$, which shows that $\Lambda$ is co-compact. 
\end{proof}

\begin{lemma} \label{lem:covol_compactopen}
Let $G$ be a locally compact Abelian group possessing a compact open subgroup $H \subseteq G$. Suppose that $\Gamma \subseteq G$ is a closed, co-compact subgroup. Then 
\begin{align} \label{eq:covolH}
\covol(\Gamma) = \# (H^{\perp} \cap \Gamma^{\perp}) \mu_{H\Gamma / \Gamma} (H\Gamma / \Gamma).
\end{align}
\end{lemma}
\begin{proof}
Since $H$ is compact and $\Gamma$ is closed, the subgroup $H \Gamma \subseteq G$ is  closed in $G$. Clearly, the subgroup $\Gamma$ is contained in $H \Gamma$, and hence an application of Lemma \ref{lem:covolumes} yields that $H \Gamma$ is co-compact in $G$. Moreover,  Lemma \ref{lem:covolumes} yields that
$\covol(\Gamma) = \mu_{H\Gamma/\Gamma}(H\Gamma /\Gamma) \mu_{G/H\Gamma}(G/H\Gamma)$. Consider the quotient group $G / H \Gamma$. Since $H$ is assumed to be open,  the group
$\Gamma H = \bigcup_{\gamma \in \Gamma} \gamma H$
is open as a union of open sets. Thus $G / H \Gamma$ is compact and discrete,  whence finite. Any finite group is self-dual and thus $G / H \Gamma \cong \widehat{G/ H \Gamma} \cong (H\Gamma)^{\perp} \cong (H^{\perp} \cap \Gamma^{\perp})$. Equipping $G/H\Gamma$ with the counting measure, it follows that
$\covol(\Gamma) = \# (H^{\perp} \cap \Gamma^{\perp}) \mu_{H\Gamma / \Gamma} (H\Gamma / \Gamma)$, which completes the proof.
\end{proof}

The previous result provides a relation between the number of points in the annihilator of a translation subgroup  intersecting a compact neighbourhood and its covolume. This relation is the essential part in the proof of Theorem \ref{thm:LIC_compactopen}. 

\begin{proof}[Proof of Theorem \ref{thm:LIC_compactopen}.]
Let $K \subseteq \widehat{G}$ be a compact set. Then there exists a finite sequence $(\omega_i)_{i =1}^n \subseteq \widehat{G}$ such that $K \subseteq \bigcup_{i = 1}^n \omega_i H^{\perp}$. Thus $K$ can be written as $K = \bigcup_{i = 1}^n \omega_i D_i$ , where $D_i := H^{\perp} \cap \omega_i^{-1} K$ for $i = 1, ...,n$. Throughout the proof, it will be assumed that $K$ is written in this form.

Suppose that $\cup_{j \in J} \{T_{\gamma} g_{j,p} \}_{\gamma \in \Gamma_j, p \in P_j}$ satisfies the local integrability condition \eqref{eq:LIC} with respect to $E = \emptyset$. Since $\bigcup_{i = 1}^n \omega_i H^{\perp}$ is compact, it follows  that
\begin{align*}
\infty &> \sum_{j \in J} \frac{1}{\covol(\Gamma_j)} \int_{P_j} \sum_{\alpha \in \Gamma_j^{\perp}} \int_{\bigcup_{i = 1}^n \omega_i H^{\perp} \cap \alpha^{-1} \bigcup_{i = 1}^n \omega_i H^{\perp}} |\hat{g}_{j,p} (\omega)|^2 \; d\mu_{\widehat{G}}(\omega) d\mu_{P_j} (p) \\
&\geq \sum_{j \in J} \frac{1}{\covol(\Gamma_j)} \int_{P_j} \sum_{\alpha \in \Gamma_j^{\perp} \cap H^{\perp}} \int_{\bigcup_{i = 1}^n \omega_i H^{\perp} \cap \alpha^{-1} \bigcup_{i = 1}^n \omega_i H^{\perp}} |\hat{g}_{j,p} (\omega)|^2 \; d\mu_{\widehat{G}}(\omega) d\mu_{P_j} (p) \\
&= \sum_{j \in J} \frac{1}{\covol(\Gamma_j)} \int_{P_j} \sum_{\alpha \in \Gamma_j^{\perp} \cap H^{\perp}} \int_{\bigcup_{i = 1}^n \omega_i H^{\perp}} |\hat{g}_{j,p} (\omega)|^2 \; d\mu_{\widehat{G}}(\omega) d\mu_{P_j} (p)  \\
&\geq \sum_{j \in J} \frac{1}{\covol(\Gamma_j)} \int_{P_j} \sum_{\alpha \in \Gamma_j^{\perp} \cap H^{\perp}} \int_{K} |\hat{g}_{j,p} (\omega)|^2 \; d\mu_{\widehat{G}}(\omega) d\mu_{P_j} (p) \\
&= \int_{K} \sum_{j \in J}\frac{1}{\covolh ( \Gamma_j)} \int_{P_j}   |\hat{g}_{j,p} (\omega)|^2 \;  d\mu_{P_j} (p) d\mu_{\widehat{G}}(\omega),
\end{align*}
where the last equality follows by the identity \eqref{eq:covolH}. 

Conversely, suppose that (ii) holds. Then \eqref{eq:covolH} yields that
\begin{align*}
\infty &>  \sum_{j \in J} \frac{1}{\covolh (\Gamma_j)}  \int_{P_j} \int_{K}   |\hat{g}_{j,p} (\omega)|^2 \;   d\mu_{\widehat{G}}(\omega) d\mu_{P_j} (p) \\
&=  \sum_{j \in J} \frac{1}{\covol (\Gamma_j)}  \int_{P_j} \sum_{\alpha \in \Gamma_j^{\perp} \cap H^{\perp}} \int_{K}   |\hat{g}_{j,p} (\omega)|^2 \;   d\mu_{\widehat{G}}(\omega) d\mu_{P_j} (p) \\
&\geq  \sum_{j \in J} \frac{1}{\covol (\Gamma_j)}  \int_{P_j} \sum_{\alpha \in \Gamma_j^{\perp} \cap H^{\perp}} \int_{K \cap \alpha^{-1} K}   |\hat{g}_{j,p} (\omega)|^2 \;   d\mu_{\widehat{G}}(\omega) d\mu_{P_j} (p) \numberthis \label{eq:2->1_LIC_compactopen}
\end{align*}
Fix $j \in J$ and consider $\{ \alpha \in \Gamma_j^{\perp} \; | \; K \cap \alpha^{-1} K \neq \emptyset\}$, which has $\Gamma_j^{\perp} \cap K^{-1} K$ as a superset. For the cardinality of this latter set, observe first that $K^{-1} K \subseteq \bigcup_{i = 1}^n \bigcup_{ \ell = 1}^n \omega_i^{-1} \omega_{\ell} H^{\perp}$ and that
\begin{align} \label{eq:sum_cardinality}
\# \bigg(\Gamma_j^{\perp} \cap \bigcup_{i = 1}^n \bigcup_{ \ell = 1}^n \omega_i^{-1} \omega_{\ell} H^{\perp} \bigg) \leq \sum_{i = 1}^n \sum_{\ell = 1}^n \# \bigg(\Gamma_j^{\perp} \cap \omega_i^{-1} \omega_{\ell} H^{\perp} \bigg).
\end{align}
For the summands in the right-hand side expression above, consider, for fixed $i, \ell \in \{1, ..., d\}$, the cases $\Gamma_j^{\perp} \cap \omega_i^{-1} \omega_{\ell} H^{\perp} = \emptyset$ and $\Gamma_j^{\perp} \cap \omega_i^{-1} \omega_{\ell} H^{\perp} \neq \emptyset$. For the first case, trivially $\# (\Gamma_j^{\perp} \cap \omega_i^{-1} \omega_{\ell} H^{\perp}) = 0$. In the latter case, let $z_{i, \ell} \in \Gamma_{j}^{\perp} \cap \omega_i^{-1} \omega_{\ell} H^{\perp}$. Then $z_{i,\ell} \Gamma_j^{\perp} = \Gamma_j^{\perp}$ and $z_{i,\ell} H^{\perp} = \omega_i^{-1} \omega_{\ell} H^{\perp}$, and thus
$\Gamma_j^{\perp} \cap \omega_i^{-1} \omega_{\ell} H^{\perp} = z_{i,j} (\Gamma_j^{\perp} \cap H^{\perp})$. Consequently, $\# (\Gamma_j^{\perp} \cap \omega_i^{-1} \omega_{\ell} H^{\perp}) = \# (\Gamma_j^{\perp} \cap H^{\perp})$. Using this, the estimate \eqref{eq:sum_cardinality} can be bounded as
\[ 
\# \bigg(\Gamma_j^{\perp} \cap \bigcup_{i = 1}^n \bigcup_{ \ell = 1}^n \omega_i^{-1} \omega_{\ell} H^{\perp} \bigg) \leq n^2 \cdot \# \bigg( \Gamma_j^{\perp} \cap  H^{\perp} \bigg),
\]
which implies, in particular, that $n^{-2} \cdot \# \{ \alpha \in \Gamma_j^{\perp} \; | \; K \cap \alpha^{-1} K \neq \emptyset\} \leq  \# (\Gamma_j^{\perp} \cap H^{\perp} )$. This, together with the estimate \eqref{eq:2->1_LIC_compactopen}, yields that
\begin{align*}
\infty &>  \sum_{j \in J} \frac{1}{\covol (\Gamma_j)}  \int_{P_j} \sum_{\alpha \in \Gamma_j^{\perp} \cap H^{\perp}} \int_{K \cap \alpha^{-1} K}   |\hat{g}_{j,p} (\omega)|^2 \;   d\mu_{\widehat{G}}(\omega) d\mu_{P_j} (p) \\
&\geq n^{-2} \; \sum_{j \in J} \frac{1}{\covol (\Gamma_j)}  \int_{P_j} \sum_{\alpha \in \Gamma_j^{\perp}} \int_{K \cap \alpha^{-1} K}   |\hat{g}_{j,p} (\omega)|^2 \;   d\mu_{\widehat{G}}(\omega) d\mu_{P_j} (p),
\end{align*}
which shows (i). 
\end{proof}

\begin{remark}
The integrability condition \eqref{eq:LIC_compactopen} in Theorem \ref{thm:LIC_compactopen} is stronger than the Calder\'on integrability condition \eqref{eq:CalderonLI}. This follows from Lemma \ref{lem:covol_compactopen}, which gives $\covolh(\Gamma_j) \leq \covol(\Gamma_j)$ for all $j \in J$. Only if the underlying group is compact  do both conditions coincide.
\end{remark}

\subsection{Euclidean space} \label{sec:Euclid}

This subsection is concerned with the hypotheses of Proposition \ref{prop:LIC_sufficient} for the case $G = \mathbb{R}^d$. In this setting, any closed, co-compact subgroup $\Gamma \subseteq \mathbb{R}^d$ has the form 
\[ \Gamma = C (\mathbb{Z}^n \times \mathbb{R}^{d-n} ) \]
for some $C \in \mathrm{GL}(d, \mathbb{R})$ and an $n \in \{0, 1, ..., d \}$. In particular, a uniform lattice is a full-rank lattice $\Gamma \subseteq \mathbb{R}^d$, that is, $\Gamma = C \mathbb{Z}^d$ for some $C \in \mathrm{GL}(d, \mathbb{R})$. 

The Haar measure on $\mathbb{R}^d$ is always assumed to be the Lebesgue measure. Using the notation $C^{\sharp} = (C^T)^{-1}$, the annihilator of a closed, co-compact subgroup $\Gamma = C (\mathbb{Z}^n \times \mathbb{R}^{d-n})$ in $\mathbb{R}^d$ reads $\Gamma^{\perp} = C^{\sharp} (\mathbb{Z}^n \times \{0\}^{d-n})$. In the sequel, the annihilator will be equipped with the weighted counting measure $|\det C|^{-1} \#$. This convention gives $\covol(\Gamma) = |\det C|$.

The next two results provide a class of translation subgroups that automatically satisfy the uniform counting estimate \eqref{eq:counting}. In particular,  the uniform counting estimate turns out to be trivial in dimension one.

\begin{lemma} \label{lem:counting_isotropic}
 Any family $(a_j (\mathbb{Z}^n \times \mathbb{R}^{d-n}))_{j \in J}$ with $a_j \in \mathbb{R}^+$ and $n \in \{0,1,...,d\}$, satisfies the uniform counting estimate \eqref{eq:counting}.
 \end{lemma}
\begin{proof}
Let $K := \overline{B_r (0)}$ be the closed Euclidean ball. First, consider the case $n = d$. 
The series in \eqref{eq:counting} reads $
 \sum_{m \in \mathbb{Z}^d} \mathds{1}_K (\omega + a_j^{-1} m)
$
for  $\omega \in \mathbb{R}^d$. For fixed $j \in J$ and $\omega \in \mathbb{R}^d$, let $m_0 \in \mathbb{Z}^d$ be such that $\mathds{1}_K (\omega + a_j^{-1} m_0) \neq 0$. If $\mathds{1}_K (\omega + a_j^{-1} m) \neq 0$ for some $m \in \mathbb{Z}^d$, then
$2r \geq \|(\omega - a^{-1}_j m) - (\omega - a_j^{-1} m_0) \| = a^{-1}_j \|m_0 - m\|$, and hence $m \in \overline{B_{a_j 2r}(m_0)}$. But since $\overline{B_{a_j 2r}(m_0)} \subseteq [m_0 - a_j 2r, m_0 + a_j 2r]^d$, the number of points $m \in \mathbb{Z}^d$ for which $\mathds{1}_K (\omega + a_j^{-1} m) \neq 0$ does not exceed
\[ 
Z_j := \# \bigg\{ m \in \mathbb{Z}^d \; \bigg| \; m \in [m_0 - a_j 2r, \; m_0 + a_j2r]^d \bigg\} \leq ( \lfloor a_j 4 r\rfloor + 1)^d. 
\]
Assume now that $r = 1/4$. Then $Z_j \leq 1$ if $a_j < 1$ and $Z_j \leq 2^d a_j^d$ if $a_j \geq 1$. Thus  $(a_j \mathbb{Z}^d)_{j \in J}$ satisfies the uniform counting estimate \eqref{eq:counting} by Lemma \ref{lem:neighbourhood}.

For the case $n < d$, the index set of the series in \eqref{eq:counting} is  $a_j^{-1}  (\mathbb{Z}^n \times \{0\}^{d-n})$. Since the inclusion $a_j^{-1}  (\mathbb{Z}^n \times \{0\}^{d-n}) \subseteq a_j^{-1} \mathbb{Z}^d$ is proper, the result follows.
\end{proof}

\begin{proposition} \label{prop:singularvalue}
Let $(C_j)_{j \in J}$ be a family of matrices $C_j \in \mathrm{GL}(d, \mathbb{R})$ such that, for all $j \in J$, the quotient of the maximal singular value of $C_j$ divided by the minimal singular value is bounded by a constant. Then $(C_j (\mathbb{Z}^n \times \mathbb{R}^{n - d}))_{j \in J}$ satisfies the uniform counting estimate \eqref{eq:counting} for any $n \in \{0,1,...,d\}$. 
\end{proposition}
\begin{proof}
For each $j \in J$, let $\sigma_{\min} (C^{\sharp}_j)$ and $\sigma_{\max} (C^{\sharp}_j)$ denote the minimal respectively the maximal singular value of $C_j^{\sharp}$. By assumption, there exists a  $K > 0$ such that
\begin{align} \label{eq:constant_singular}
\sup_{j \in J} \; \frac{\sigma_{\max}(C_j^{\sharp})} {\sigma_{\min} (C_j^{\sharp})} \leq K.
\end{align}
 Write $C_j^{\sharp} = L_j b_j I$, where $b_j := \sigma_{\min} (C_j^{\sharp})$ and $L_j := b_j^{-1} C_j^{\sharp}$.  Then $\sigma_{\max} (L_j) \leq K$ and $\sigma_{\min} (L_j) = 1$ for all $j \in J$. In order to verify the uniform counting estimate \eqref{eq:counting}, fix an $r > 0$ and let $B_r (\omega)$  denote the Euclidean ball with centre $\omega \in \widehat{\mathbb{R}}^d$. Then
\[ L_j^T (B_r (\omega)) \subseteq B_{Kr} (L_j^T \omega) \]
 for any $j \in J$ and $\omega \in \widehat{\mathbb{R}}^d$. According to Lemma \ref{lem:counting_isotropic}, the family $(b_j^{-1} (\mathbb{Z}^n \times \mathbb{R}^{d-n}))_{j \in J}$ satisfies the uniform counting estimate \eqref{eq:counting}. Hence, there exists a constant $C > 0$ such that
\begin{align} \label{eq:singular_UCE}
 \sup_{\omega \in \mathbb{R}^d} \# \bigg( b_j ( \mathbb{Z}^n \times \{0 \}^{d-n}) \; \cap \; B_{Kr}(L^T_j \omega)  \bigg) \leq C \bigg( 1 + b_j^{-d} \bigg). 
\end{align}
for all $j \in J$. By use of the assumption \eqref{eq:constant_singular}, it follows that 
\[ b_j^{-d} = \sigma_{\max} (C_j)^{d} \leq K^d \sigma_{\min} (C_j)^d \leq K^d |\det C_j|. \] This, together with the estimate \eqref{eq:singular_UCE}, yields that 
\begin{align*}
\sup_{\omega \in \mathbb{R}^d} \# \bigg( C_j^{\sharp} ( \mathbb{Z}^n \times \{0 \}^{d-n}) \; \cap \;  B_{r} (\omega)  \bigg) &\leq  \sup_{\omega \in \mathbb{R}^d} \# \bigg( b_j ( \mathbb{Z}^n \times \{0 \}^{d-n}) \; \cap \; B_{Kr}(L^T_j \omega)  \bigg) \\
&\leq C K^d \bigg( 1 + | \det C_j | \bigg)
\end{align*}
for all $j \in J$. An application of Lemma \ref{lem:neighbourhood} therefore completes the proof.
\end{proof}

For a family $(C_j)_{j \in J} \subseteq \mathrm{GL}(d,\mathbb{R})$ that does not satisfy the hypotheses of Proposition \ref{prop:singularvalue}, the lattice system $(C_j \mathbb{Z}^d)_{j \in J}$ might fail the uniform counting estimate as demonstrated by Example \ref {ex:failingUCEbutLIC}. An existence result for generalised shift-invariant frames with lattice systems satisfying the hypotheses in Proposition \ref{prop:singularvalue} is contained in \cite{fuhr2017system}. 

The next result identifies a class of lattice systems for which the associated system satisfies the temperateness condition \eqref{eq:norm-bounded} provided that the system satisfies the local integrability condition.  

\begin{proposition} \label{prop:diagonal}
For $j \in J$, let $C_j := \diag(a_1^{(j)}, ..., a_d^{(j)}) \in \mathrm{GL}(d, \mathbb{R})$ for $a_i^{(j)} \in \mathbb{R}^+$, and let $\Gamma_j := C_j \mathbb{Z}^d$. Suppose the generalised translation-invariant system $\cup_{j \in J} \{T_{\gamma} g_{j,p} \}_{\gamma \in \Gamma_j, p \in P_j}$ satisfies the local integrability condition \eqref{eq:LIC} with respect to $E \in \mathcal{E}$. Then the generating functions $\cup_{j \in J} \{g_{j,p} \}_{p \in P_j}$ are strictly temperate \eqref{eq:norm-bounded} with respect to $E$.
\end{proposition}
\begin{proof}
By use of Lemma \ref{lem:lowerCE}, it suffices to show the estimate \eqref{eq:lowerCE}. In order to do so, fix an $E \in \mathcal{E}$ and let $K \subseteq \widehat{\mathbb{R}}^d \setminus E$ be a compact set. Choose $r > 2$ such that $K \subseteq [-r,r]^d \setminus E$  and let $Q := [-r,r]^d \setminus E$.
For fixed $j \in J$, consider the series 
\[
 \sum_{\alpha \in \Gamma_j^{\perp}} \mathds{1}_{Q} (\omega + \alpha) =\sum_{k \in \mathbb{Z}^d} \mathds{1}_{Q} (\omega + C^{-1}_j k)  = \sum_{k \in \mathbb{Z}^d} \mathds{1}_{C_j (Q)} (C_{j} \omega +  k),
\]
for $\omega \in Q$. Observe that $C_j (Q) = \prod_{i =1}^d [-a_{i}^{(j)} r, a_{i}^{(j)} r]$ for each $j \in J$. Therefore
\[
\sum_{\alpha \in \Gamma_j^{\perp}} \mathds{1}_{Q} (\omega + \alpha) = \prod_{i = 1}^d \# \bigg\{ m \in \mathbb{Z}  \; \bigg| \; m \in [-a_i^{(j)}r - a_i^{(j)} \omega_i, \; a_i^{(j)}r + a_i^{(j)} \omega_i] \bigg\}
\]
for $\omega = (\omega_1, ..., \omega_d) \in Q$. Since $r > 2$, for any fixed $i = 1, ..., d$ and $j \in J$, there exists a constant $K_i > 0$ such that
\[
\# \bigg\{ m \in \mathbb{Z}  \; \bigg| \; m \in [-a_i^{(j)} (r -  \omega_i) , \;a_i^{(j)}(r + \omega_i) ] \bigg\} \geq K_i a_i^{(j)}.
\]
Therefore, it follows that, for all $j \in J$,
\[
\inf_{\omega \in Q} \sum_{\alpha \in \Gamma_j^{\perp}} \mathds{1}_{Q} (\omega + \alpha) \geq \prod_{i = 1}^d K_i |\det C_j|,
\]
which shows that condition \eqref{eq:lowerCE} is satisfied. 
\end{proof}

\begin{corollary} \label{cor:LIC_line}
Let $(C_j)_{j \in J}$ be a family of diagonal matrices $C_j \in \mathrm{GL}(d, \mathbb{R})$ such that, for all $j \in J$, the quotient of the maximal singular value of $C_j$ divided by the minimal singular value is bounded by a constant. Then the following assertions are equivalent:
\begin{enumerate}[(i)]
\item The system $\cup_{j \in J} \{T_{\gamma} g_{j,p} \}_{\gamma \in C_j \mathbb{Z}^d, p \in P_j}$ satisfies the local integrability condition \eqref{eq:LIC}. 
\item The system $\cup_{j \in J} \{T_{\gamma} g_{j,p} \}_{\gamma \in C_j \mathbb{Z}^d, p \in P_j}$ satisfies the Calder\'on integrability condition \eqref{eq:CalderonLI} and the temperateness condition \eqref{eq:norm-bounded}.
\end{enumerate}
\end{corollary}
\begin{proof}
That (i) implies (ii) follows directly by combining Lemma \ref{lem:LIC->CalderonLI} and Proposition \ref{prop:diagonal}. The converse implication is a consequence of Proposition \ref{prop:LIC_sufficient} and Proposition \ref{prop:singularvalue}.
\end{proof}
 
For $G = \mathbb{R}$, the previous results recovers \cite[Proposition 3.3]{christensen2017explicit}. Both implications might fail for lattice systems that do not satisfy the hypotheses, cf.  Example \ref{ex:mainexample} and Remark \ref {rem:construction_wavelet}.

The remainder of this section is devoted to lattice systems $(\Gamma_j )_{j \in \mathbb{Z}}$ whose elements form the image of a closed, co-compact subgroup under an integer power of a non-singular matrix, that is, $\Gamma_j := A^j \Gamma$ for some closed, co-compact subgroup $\Gamma \subseteq \mathbb{R}^d$ and an $A \in \mathrm{GL}(d, \mathbb{R})$. Counting estimates for lattice systems of this form have appeared in numerous papers, including \cite{hernandez2002unified,chui2002characterization,calogero2000characterization,guo2006some,bownik2011affine}. In particular, a detailed study of the counting estimate was given by Bownik \& Lemvig \cite{bownik2017wavelets}. In all these papers, the notion of an expanding matrix or, more generally, a matrix expanding on a subspace plays a prominent role.

\begin{definition}
A matrix $A \in \mathrm{GL}(d, \mathbb{R})$ is said to be \emph{expanding on a subspace} if
\begin{enumerate}[(a)]
\item All eigenvalues $\lambda \in \sigma(A)$ satisfy $|\lambda| \geq 1$.
\item There exists an eigenvalue $\lambda \in \sigma(A)$ satisfying $|\lambda|>1$.
\item All eigenvalues $\lambda \in \sigma(A)$ satisfying $|\lambda|=1$  have Jordan blocks of order one. 
\end{enumerate}
A matrix $A \in \mathrm{GL}(d, \mathbb{R})$ whose eigenvalues $\lambda \in \sigma(A)$ satisfy $|\lambda| > 1$ is called \emph{expanding}.
\end{definition} 

The next result is \cite[Theorem 2.4]{bownik2017wavelets}. For the first part of the theorem, see also the corresponding results in \cite{hernandez2002unified,guo2006some,bownik2011affine}.

\begin{theorem}[\cite{bownik2017wavelets}] \label{thm:LCE_bownik}
Let $A \in \mathrm{GL}(d, \mathbb{R})$ be such that $|\det A| > 1$.
\begin{enumerate}[(i)]
\item Suppose that $A$ is expanding on a subspace. Then, for every full-rank lattice $\Gamma \subseteq \mathbb{R}^d$, the pair $(A, \Gamma)$ satisfies the \emph{lattice counting estimate}, i.e., there exists a constant $C > 0$ such that
\begin{align} \label{eq:LCE}
\# \bigg( \Gamma \cap A^j B_r (0) \bigg) \leq C \max\{1, |\det A |^j \}
\end{align}
for all $j \in \mathbb{Z}$.
\item Suppose the pair $(A, \Gamma)$ satisfies the lattice counting estimate \eqref{eq:LCE} for all full-rank lattices $\Gamma \subseteq \mathbb{R}^d$. Then $A$ is expanding on a subspace. 
\end{enumerate}
\end{theorem}

The statements in the previous result hold for all full-rank lattices $\Gamma \subseteq \mathbb{R}^d$. In case an arbitrary matrix $A \in \mathrm{GL}(d, \mathbb{R})$ with $|\det A|>1$ is fixed, there might still exist a full-rank lattice $\Gamma \subseteq \mathbb{R}^d$ such that $(A^j \Gamma)_{j \in \mathbb{Z}}$ satisfies the lattice counting estimate \eqref{eq:LCE}. For example, the system $(A^j \mathbb{Z}^3)_{j \in \mathbb{Z}}$ in $\mathbb{R}^3$, where
\begin{align} \label{eq:guo_example}
 A
=
\begin{pmatrix}
a & 0 & 0 \\
0 & 1 & 1 \\
0 & 0 & 1
\end{pmatrix} 
\end{align}
for $a \in \mathbb{R}$ with $|a| > 1$, satisfies the lattice counting estimate \cite[Example 4]{guo2006some}, but $A$ is not expansive on a subspace. This example should be considered in the light of  \cite[Corollary 4.7]{bownik2017wavelets}, which asserts that given any $A \in \mathrm{GL}(d, \mathbb{R})$ with $|\det A| > 1$, then the pair $(A, \Gamma)$ satisfies the lattice counting estimate \eqref{eq:LCE} for almost all unimodular lattices $\Gamma \subseteq \mathbb{R}^d$. Moreover, given any full-rank lattice $\Gamma \subseteq \mathbb{R}^d$, then the pair $(A, \Gamma)$ satisfies the lattice counting estimate \eqref{eq:LCE} for almost all $A \in \mathrm{GL}(d, \mathbb{R})$ with $|\det A|>1$, cf. \cite{bownik2017wavelets}. 

The following result is a simple consequence of Theorem \ref{thm:LCE_bownik}.  

\begin{corollary} \label{cor:LCE_UCE}
Let $\Gamma \subseteq \mathbb{R}^d$ be a closed, co-compact subgroup and let $A \in \mathrm{GL}(d, \mathbb{R})$ be such that $|\det A| > 1$ and that $A$ is expanding on a subspace. Then the family $( \Gamma_j)_{j \in \mathbb{Z}}$ with $\Gamma_j = A^j \Gamma$ satisfies the uniform counting estimate \eqref{eq:counting}.
\end{corollary} 
\begin{proof}
 Write $\Gamma = C (\mathbb{Z}^n \times \mathbb{R}^{d-n} )$ for some $n \in \{0,1,...,d\}$ and $C \in \mathrm{GL}(d, \mathbb{R})$. Define the full-rank lattice $\Lambda = C \mathbb{Z}^d$.  Take $A \in \mathrm{GL} (d, \mathbb{R})$  such that it is expanding on a subspace. Then $A^T$ is also expanding on a subspace. Define $\Gamma_j = A^j \Gamma$ and $\Lambda_j = A^j \Lambda$ for $j \in \mathbb{Z}$. Since $\Lambda \subseteq \Gamma$, it follows that $\Lambda_j \subseteq \Gamma_j$ for all $j \in \mathbb{Z}$, which in turn implies that $\Gamma_j^{\perp} \subseteq \Lambda_j^{\perp}$ for all $j \in \mathbb{Z}$. Using Theorem \ref{thm:LCE_bownik}(i) for the pair $(A^T, \Lambda^{\perp})$, there exists a constant $C > 0$ such that 
\begin{align*}
\# \bigg( \Lambda_j^{\perp} \cap B_r (0) \bigg) &=  \# \bigg( (A^T)^{-j} \Lambda^{\perp} \cap B_r (0) \bigg) =  \# \bigg( \Lambda^{\perp} \cap (A^T)^{j} B_r (0) \bigg) \leq C\bigg(1 + |\det A|^{j} \bigg) 
\end{align*}
for all $j \in \mathbb{Z}$. Hence, fixing a compact set $K \subseteq \mathbb{R}^d$ and choosing $r>0$ such that $K \subseteq B_r (0)$, it follows
\[
\# \bigg( \Gamma_j^{\perp} \cap K \bigg) \leq \# \bigg( \Lambda_j^{\perp} \cap B_r (0) \bigg) \leq C (1 + \covol(\Gamma_j)),
\]
as required. 
\end{proof}

A counting estimate for possibly non-discrete, but closed, co-compact subgroups is also contained in \cite{barbieri2017calderon}, where it is called \emph{Property X}. 
In fact, the results in \cite{barbieri2017calderon} hold for expansive automorphisms on general locally compact Abelian groups. See \cite[Appendix A]{barbieri2017calderon} for the relation between Property X and the lattice counting estimate
for expansive matrices as considered here.

\begin{remark} \label{rem:construction_wavelet}
Let $a \in \; ]0,1[$ be such that $\frac{a^2}{a-1} \in \mathbb{R} \setminus \mathbb{Q}$. Define
\[
B = 
\begin{pmatrix}
2 & 0 & 0 \\
0 & 1 + a & 1 \\
0 & -a^2 & 1-a 
\end{pmatrix}
\in \mathrm{GL}(3, \mathbb{R})
\]
and set $A := B^T$. By \cite[Lemma 5.4]{guo2006some}, the pair $(B, \mathbb{Z})$ fails the lattice counting estimate \eqref{eq:LCE}. Based on the pair $(B, \mathbb{Z})$, one can construct a function $\psi \in L^2 (\mathbb{R}^3)$ such that $\{T_{\gamma} D_{A^j} \psi \}_{j \in \mathbb{Z}, \gamma \in A^{-j} \mathbb{Z}}$, where $D_{A^j} \psi := |\det A|^{j/2} \psi (A^j \cdot)$, fails the local integrability condition, but satisfies
\begin{align} \label{eq:Calderon_wavelets}
\sum_{j \in \mathbb{Z}} \frac{1}{|\det A|^j} |\widehat{D_{A^j} \psi}(\omega)|^2 = \sum_{j \in \mathbb{Z}} |\psi(B^{-j} \omega)|^2 = |\psi (\omega)|^2 \leq 1
\end{align}
for $\mu_{\widehat{\mathbb{R}}^3}$-a.e. $\omega \in \widehat{\mathbb{R}}^3$ by \cite[Theorem 5.1]{guo2006some}. Thus $\{T_{\gamma} D_{A^j} \psi \}_{j \in \mathbb{Z}, \gamma \in A^{-j} \mathbb{Z}}$ satisfies the Calder\'on integrability condition \eqref{eq:CalderonLI} and the strictly temperateness condition \eqref{eq:norm-bounded}. 

The construction of this example works in fact for any pair $(B, \Gamma)$, where $B \in \mathrm{GL}(d, \mathbb{R})$ with $|\det B|>1$ and $\Gamma \subseteq \mathbb{R}^d$ is a full-rank lattice, that fails the lattice counting estimate \eqref{eq:LCE}. See the proof of \cite[Theorem 3.8]{bownik2017wavelets}.
\end{remark}

The following fact is used implicitly in wavelet theory.

\begin{lemma} \label{lem:temperate_wavelet}
Let $A \in \mathrm{GL}(d, \mathbb{R})$ be expanding on a subspace and let $P$ be a finite set.  Then, for any family $\{g_p\}_{p \in P}$ in $L^2 (\mathbb{R}^d)$, the dilations $ \{|\det A|^{j/2}g_{p} ( A^j \cdot)\}_{j \in \mathbb{Z}, p \in P}$ satisfy the temperateness condition \eqref{eq:norm-bounded}. 
\end{lemma}
\begin{proof}
Confer the proof of \cite[Proposition 5.12]{hernandez2002unified}.
\end{proof}

The next result provides a complete characterisation of the local integrability condition \eqref{eq:LIC} for regular wavelet systems whose dilation matrix is expanding on a subspace. The result recovers \cite[Theorem 3.8]{bownik2017wavelets}.  Here, the assumption that $|\det A|>1$ is not a real restriction since for the case $|\det A|=1$ wavelet frames do not exist \cite{laugesen2002characterization,larson2006explicit}.

\begin{proposition}
Let $A \in \mathrm{GL}(d, \mathbb{R})$ be expanding on a subspace with $|\det A|>1$. Let $\Gamma \subseteq \mathbb{R}^d$ be a closed, co-compact subgroup and let $P$ be a finite set. Then the following assertions are equivalent:
\begin{enumerate}[(i)]
\item The system $\cup_{j \in \mathbb{Z}} \{T_{\gamma} D_{A^j} g_p \}_{\gamma \in A^{-j} \Gamma, p \in P}$ satisfies the local integrability condition \eqref{eq:LIC}.
\item The system $\cup_{j \in \mathbb{Z}} \{T_{\gamma} D_{A^j} g_p \}_{ \gamma \in A^{-j}  \Gamma, p \in P}$ satisfies the Calder\'on integrability condition \eqref{eq:CalderonLI}. 
\end{enumerate}
\end{proposition}
\begin{proof}
The fact that (i) implies (ii) is simply Lemma \ref{lem:LIC->CalderonLI}. The reverse implication follows by combining Proposition \ref{prop:LIC_sufficient}, Corollary \ref{cor:LCE_UCE} and Lemma \ref{lem:temperate_wavelet}.
\end{proof}

\subsection*{Acknowledgements} The author gratefully acknowledges support from the Austrian Science Fund (FWF):P29462-N35. The author thanks Jos\'e Luis Romero for useful discussions and for his help with several of the examples. Thanks also goes to Peter Kuleff and Jakob Lemvig for reading the manuscript and providing helpful comments.


\begin{thebibliography}{10}

\bibitem{balan2014multi}
R.~Balan, J.~G. Christensen, I.~A. Krishtal, K.~A. Okoudjou, and J.~L. Romero.
\newblock Multi-window {G}abor frames in amalgam spaces.
\newblock {\em Math. Res. Lett.}, 21(1):55--69, 2014.

\bibitem{barbieri2017calderon}
D.~Barbieri, E.~Hern\'{a}ndez, and A.~Mayeli.
\newblock Calder{\'o}n-type inequalities for affine frames.
\newblock Preprint, arXiv:1706.06518.

\bibitem{barbieri2015zak}
D.~Barbieri, E.~Hern\'{a}ndez, and V.~Paternostro.
\newblock The {Z}ak transform and the structure of spaces invariant by the
  action of an {LCA} group.
\newblock {\em J. Funct. Anal.}, 269(5):1327--1358, 2015.

\bibitem{benedetto2004wavelet}
J.~J. Benedetto and R.~L. Benedetto.
\newblock A wavelet theory for local fields and related groups.
\newblock {\em J. Geom. Anal.}, 14(3):423--456, 2004.

\bibitem{bownik2011affine}
M.~Bownik and J.~Lemvig.
\newblock Affine and quasi-affine frames for rational dilations.
\newblock {\em Trans. Amer. Math. Soc.}, 363(4):1887--1924, 2011.

\bibitem{bownik2017wavelets}
M.~Bownik and J.~Lemvig.
\newblock Wavelets for non-expanding dilations and the lattice counting
  estimates.
\newblock {\em Int. Math. Res. Not.}, 2017(23):7264--7291, 2017.

\bibitem{bownik2015structure}
M.~Bownik and K.~A. Ross.
\newblock The structure of translation-invariant spaces on locally compact
  abelian groups.
\newblock {\em J. Fourier Anal. Appl.}, 21(4):849--884, 2015.

\bibitem{bownik2004spectral}
M.~Bownik and Z.~Rzeszotnik.
\newblock The spectral function of shift-invariant spaces on general lattices.
\newblock In {\em Wavelets, frames and operator theory}, volume 345 of {\em
  Contemp. Math.}, pages 49--59. Amer. Math. Soc., Providence, RI, 2004.

\bibitem{calogero2000characterization}
A.~Calogero.
\newblock A characterization of wavelets on general lattices.
\newblock {\em J. Geom. Anal.}, 10(4):597--622, 2000.

\bibitem{christensen2016introduction}
O.~Christensen.
\newblock {\em An introduction to frames and {R}iesz bases}.
\newblock Applied and Numerical Harmonic Analysis. Birkh\"auser, second
  edition, 2016.

\bibitem{christensen2005generalized}
O.~Christensen and Y.~C. Eldar.
\newblock Generalized shift-invariant systems and frames for subspaces.
\newblock {\em J. Fourier Anal. Appl.}, 11(3):299--313, 2005.

\bibitem{christensen2017explicit}
O.~Christensen, M.~Hasannasab, and J.~Lemvig.
\newblock Explicit constructions and properties of generalized shift-invariant
  systems in {$L^2(\Bbb R)$}.
\newblock {\em Adv. Comput. Math.}, 43(2):443--472, 2017.

\bibitem{chui2002characterization}
C.~K. Chui, W.~Czaja, M.~Maggioni, and G.~Weiss.
\newblock Characterization of general tight wavelet frames with matrix
  dilations and tightness preserving oversampling.
\newblock {\em J. Fourier Anal. Appl.}, 8(2):173--200, 2002.

\bibitem{fuhr2010generalized}
H.~F\"uhr.
\newblock Generalized {C}alder\'on conditions and regular orbit spaces.
\newblock {\em Colloq. Math.}, 120(1):103--126, 2010.

\bibitem{fuhr2015coorbit}
H.~F\"uhr.
\newblock Coorbit spaces and wavelet coefficient decay over general dilation
  groups.
\newblock {\em Trans. Amer. Math. Soc.}, 367(10):7373--7401, 2015.

\bibitem{fuhr2017system}
H.~F\"uhr and J.~Lemvig.
\newblock System bandwidth and the existence of generalized shift-invariant
  frames.
\newblock {\em J. Funct. Anal.}, 276(2):563--601, 2019.

\bibitem{guo2006some}
K.~Guo and D.~Labate.
\newblock Some remarks on the unified characterization of reproducing systems.
\newblock {\em Collect. Math.}, 57(3):295--307, 2006.

\bibitem{hernandez2002unified}
E.~Hern\'andez, D.~Labate, and G.~Weiss.
\newblock A unified characterization of reproducing systems generated by a
  finite family. {II}.
\newblock {\em J. Geom. Anal.}, 12(4):615--662, 2002.

\bibitem{hewitt1963abstract}
E.~Hewitt and K.~A. Ross.
\newblock {\em Abstract harmonic analysis. {V}ol. {I}: {S}tructure of
  topological groups. {I}ntegration theory, group representations}.
\newblock Die Grundlehren der mathematischen Wissenschaften, Band 115.
  Springer-Verlag, 1963.

\bibitem{hewitt1970abstract}
E.~Hewitt and K.~A. Ross.
\newblock {\em Abstract harmonic analysis. {V}ol. {II}: {S}tructure and
  analysis for compact groups. {A}nalysis on locally compact {A}belian groups}.
\newblock Die Grundlehren der mathematischen Wissenschaften, Band 152.
  Springer-Verlag, 1970.

\bibitem{iverson2015subspaces}
J.~W. Iverson.
\newblock Subspaces of {$L^2(G)$} invariant under translation by an abelian
  subgroup.
\newblock {\em J. Funct. Anal.}, 269(3):865--913, 2015.

\bibitem{jakobsen2016reproducing}
M.~S. Jakobsen and J.~Lemvig.
\newblock Reproducing formulas for generalized translation invariant systems on
  locally compact abelian groups.
\newblock {\em Trans. Amer. Math. Soc.}, 368(12):8447--8480, 2016.

\bibitem{kutyniok2006local}
G.~Kutyniok.
\newblock The local integrability condition for wavelet frames.
\newblock {\em J. Geom. Anal.}, 16(1):155--166, 2006.

\bibitem{kutyniok2006theory}
G.~Kutyniok and D.~Labate.
\newblock The theory of reproducing systems on locally compact abelian groups.
\newblock {\em Colloq. Math.}, 106(2):197--220, 2006.

\bibitem{labate2004approach}
D.~Labate, G.~Weiss, and E.~Wilson.
\newblock An approach to the study of wave packet systems.
\newblock In {\em Wavelets, frames and operator theory}, volume 345 of {\em
  Contemp. Math.}, pages 215--235. Amer. Math. Soc., Providence, RI, 2004.

\bibitem{lagarias1991bounds}
J.~C. Lagarias and G.~M. Ziegler.
\newblock Bounds for lattice polytopes containing a fixed number of interior
  points in a sublattice.
\newblock {\em Canad. J. Math.}, 43(5):1022--1035, 1991.

\bibitem{larson2006explicit}
D.~Larson, E.~Schulz, D.~Speegle, and K.~F. Taylor.
\newblock Explicit cross-sections of singly generated group actions.
\newblock In {\em Harmonic analysis and applications}, Appl. Numer. Harmon.
  Anal., pages 209--230. Birkh\"auser Boston, Boston, MA, 2006.

\bibitem{laugesen2001completeness}
R.~S. Laugesen.
\newblock Completeness of orthonormal wavelet systems for arbitrary real
  dilations.
\newblock {\em Appl. Comput. Harmon. Anal.}, 11(3):455--473, 2001.

\bibitem{laugesen2002translational}
R.~S. Laugesen.
\newblock Translational averaging for completeness, characterization and
  oversampling of wavelets.
\newblock {\em Collect. Math.}, 53(3):211--249, 2002.

\bibitem{laugesen2002characterization}
R.~S. Laugesen, N.~Weaver, G.~L. Weiss, and E.~N. Wilson.
\newblock A characterization of the higher dimensional groups associated with
  continuous wavelets.
\newblock {\em J. Geom. Anal.}, 12(1):89--102, 2002.

\bibitem{lemvig2017sufficient}
J.~Lemvig and J.~T. Van~Velthoven.
\newblock Criteria for generalised translation-invariant frames.
\newblock {\em Stud. Math.}, to appear.

\bibitem{reiter2000classical}
H.~Reiter and J.~D. Stegeman.
\newblock {\em Classical harmonic analysis and locally compact groups},
  volume~22 of {\em London Mathematical Society Monographs}.
\newblock The Clarendon Press, Oxford University Press, New York, second
  edition, 2000.

\bibitem{ron2005generalized}
A.~Ron and Z.~Shen.
\newblock Generalized shift-invariant systems.
\newblock {\em Constr. Approx.}, 22(1):1--45, 2005.

\bibitem{rudin1962fourier}
W.~Rudin.
\newblock {\em Fourier analysis on groups}.
\newblock Interscience Tracts in Pure and Applied Mathematics, No. 12.
  Interscience Publishers, New York-London, 1962.

\bibitem{tao2006additive}
T.~Tao and V.~Vu.
\newblock {\em Additive combinatorics}, volume 105 of {\em Cambridge Studies in
  Advanced Mathematics}.
\newblock Cambridge University Press, Cambridge, 2006.

\end{thebibliography}
\end{document}